\documentclass[12pt]{amsart}
\usepackage{a4wide,graphicx}
\usepackage{color,enumerate}
\allowdisplaybreaks

\usepackage{caption}
\usepackage[labelformat=simple,labelfont={}]{subcaption}

\let\pa\partial 
\let\na\nabla  
\let\eps\varepsilon  
\newcommand{\N}{{\mathbb N}}  
\newcommand{\R}{{\mathbb R}} 
\newcommand{\diver}{\operatorname{div}}  
\newcommand{\T}{{\mathbb T}}

\newcommand{\V}{{\mathcal V}}
\newcommand{\HH}{{\mathcal H}}
\newcommand{\X}{{\mathcal X}}
\newcommand{\Z}{{\mathcal Z}}

\newtheorem{theorem}{Theorem}   
\newtheorem{lemma}[theorem]{Lemma}   
\newtheorem{proposition}[theorem]{Proposition}   
\newtheorem{remark}[theorem]{Remark}

\begin{document}  

\title[Entropy dissipative one-leg schemes]{Entropy dissipative 
one-leg multistep time approximations of nonlinear diffusive equations}

\author{Ansgar J\"ungel}
\address{Institute for Analysis and Scientific Computing, Vienna University of  
	Technology, Wiedner Hauptstra\ss e 8--10, 1040 Wien, Austria}
\email{juengel@tuwien.ac.at} 

\author[J. P. Mili\v{s}i\'{c}]{Josipa-Pina Mili\v{s}i\'{c}}
\address{Department of Applied Mathematics,
Faculty of Electrical Engineering and Computing, University of Zagreb, Unska 3,
10000 Zagreb, Croatia}
\email{pina.milisic@fer.hr}

\date{\today}

\thanks{The first author acknowledges partial support from   
the Austrian Science Fund (FWF), grants P22108, P24304, I395, and W1245,
and the Austrian-French Project of the Austrian Exchange Service (\"OAD).
This research was supported by the European Union under 
Grant Agreement number 304617 (Marie-Curie Project ``Novel Methods in 
Computational Finance'')} 

\begin{abstract}
New one-leg multistep time discretizations of nonlinear 
evolution equations are investigated. The main features of the scheme
are the preservation of the nonnegativity and the entropy-dissipation
structure of the diffusive equations. The key ideas are to combine Dahlquist's
G-stability theory with entropy-dissipation methods and to introduce a
nonlinear transformation of variables which provides a quadratic structure
in the equations.
It is shown that G-stability of the one-leg scheme is sufficient to derive
discrete entropy dissipation estimates. The general result is applied to
a cross-diffusion system from population dynamics and
a nonlinear fourth-order quantum diffusion model,
for which the existence of semi-discrete weak solutions is proved.
Under some assumptions on the operator of the evolution equation, the
second-order convergence of solutions is shown. Moreover, some numerical
experiments for the population model are presented, which underline the theoretical
results.
\end{abstract}

\keywords{Linear multistep methods, entropy dissipation, diffusion equations, 
population dynamics, quantum drift-diffusion equation, 
Derrida-Lebowitz-Speer-Spohn equation, existence of solutions.}  
 
\subjclass[2000]{65M12, 35Q40, 92D25, 82D37.}

\maketitle


\section{Introduction}

Evolution equations with applications in the natural sciences typically contain some
structural information reflecting inherent physical properties such as
positivity, mass conservation, or energy and entropy dissipation.
In this paper, we propose novel one-step and two-step semidiscrete numerical schemes, 
which preserve the structure of the underlying diffusive equations.
For the analysis, we combine linear multistep discretizations, investigated for
ordinary differential equations from the 1980s on, and entropy dissipation
methods, which have been proposed in recent years.

Linear multistep methods refer to previous time steps and derivative values.
They are proposed to solve stiff differential equations. An important class
of these methods are the backward differentiation formulas (BDF). 
Multistep methods were also applied
to nonlinear evolution equations. For instance, linear multistep schemes
for fully nonlinear problems, which are governed by a nonlinear
mapping with sectorial first Fr\'echet derivative, were
dealt with, by linearization, in \cite{GOPT02,OTK04}, and quasilinear evolution
equations were treated in \cite{LeR80}. In \cite{Han06,Rul96},
multistep discretizations 
for problems governed by maximal monotone or monotone operators 
were studied. For monotone evolution equations, also
other schemes were proposed, for instance, stiffly accurate implicit 
Runge-Kutta methods \cite{EmTh10}.
Error estimates for two-step BDF methods for nonlinear evolution equations
were shown in \cite{Emm05,Emm09}.

Dahlquist introduced in \cite{Dah76} so-called one-leg methods which need
only one function evaluation in each time step. 
With every multistep method we can associate its one-leg counterpart and vice versa.
It turned out that one-leg methods allow for a stability analysis
for stiff nonlinear problems. The stability behavior was defined
in terms of the so-called G-stability \cite{Dah78}, which can be related
to discrete ``energy'' dissipation \cite{Hil97}. Stiffness independent
error estimates with the optimal order of convergence were derived in \cite{HuSt91}.

Our schemes are generalizations of dissipative multistep methods analyzed in,
e.g., \cite{DLN83,Hil97,Hua00}. In order to explain the idea,
we consider the evolution equation
\begin{equation}\label{1.eq}
  u_t + A(u) = 0, \quad t>0, \quad u(0)=u_0,
\end{equation}
where $A$ is some (nonlinear) operator defined on $D(A)\subset\V$, where 
$\V\hookrightarrow \HH\hookrightarrow\V'$ is a Gelfand tripel
(see Section \ref{sec.thm3} for details). 
In the literature, usually the monotonicity condition
\begin{equation}\label{1.mono}
  \langle A(u),u\rangle \ge 0 \quad\mbox{for all }u\in D(A),
\end{equation}
where $\langle\cdot,\cdot\rangle$ is the dual product between $\V'$ and $\V$, 
is assumed \cite{Hil97,HuSt91}.
This condition implies that the ``energy'' $\frac12\|u(t)\|^2$ is nonincreasing:
\begin{equation}\label{1.ene}
  \frac12\,\frac{d}{dt}\|u(t)\|^2 = \langle u_t(t),u(t)\rangle 
  = -\langle A(u(t)),u(t)\rangle \le 0,
\end{equation}
where $\|\cdot\|$ is the norm on $\HH$.

In many situations, not the ``energy'' is increasing
but a nonlinear expression $H[u(t)]$, which we call an ``entropy''. 
To explain this statement, we identify the Fr\'echet derivative $H'[u]$ with its
Riesz representative $h'(u)$, i.e.\ $H'[u]v=(h'(u),v)$, 
where $(\cdot,\cdot)$ is the scalar product on $\HH$. 
Then, replacing assumption \eqref{1.mono} by
\begin{equation}\label{1.monoh}
  \langle A(u),h'(u)\rangle \ge 0\quad\mbox{for all }u\in D(A),
\end{equation}
we find that formally
\begin{equation}\label{1.ent}
  \frac{d}{dt}H[u(t)] = (u_t(t),h'(u(t))) = -\langle A(u(t)),h'(u(t))\rangle \le 0,
\end{equation}
i.e., $H[u]$ is a Lyapunov functional for \eqref{1.eq} which expresses the
dissipation property of the evolution equation. Note that condition \eqref{1.monoh}
formally reduces to \eqref{1.mono} for the special choice $H[u]=\frac12\|u\|^2$.

In order to recover property \eqref{1.ene} on a discrete level, Hill \cite{Hil97}
discretizes \eqref{1.eq} by the one-leg method
$$
  \tau^{-1}\rho(E)u_k + A(\sigma(E)u_k) = 0,
$$
where $u_k$ approximates $u(t_k)$ with $t_k=\tau k$, $\tau>0$ is the time step size,
and
\begin{equation}\label{1.rho}
   \tau^{-1}\rho(E)u_k =  \tau^{-1}\sum_{j=0}^p \alpha_j u_{k+j}, \quad
  \sigma(E)u_k = \sum_{j=0}^p\beta_j u_{k+j}
\end{equation}
with $\alpha_j$, $\beta_j\in\R$
are approximations of $u_t(t_k)$ and $u(t_k)$, respectively. 
Hill proves that, under some assumptions, this scheme yields 
a dissipative discretization if and only if the scheme is strongly A-stable
or strongly G-stable (see Section \ref{sec.g.ene} for a definition of G-stability). 
The proof relies on the quadratic structure of the definition of G-stability; this
structure is already present in \eqref{1.ene}. 
Unfortunately, \eqref{1.ent} does not possess such a structure and the proofs of
\cite{Hil97} do not apply under the assumption \eqref{1.monoh}.

The main idea of this paper is to {\em enforce} a quadratic structure
by introducing the variable $v$ by $v^2=h(u)$ (assuming that $h(u)\ge 0$).
More precisely, we discretize \eqref{1.eq} in the formulation
\begin{equation}\label{1.eqv}
  h(u)^{1/2}h'(u)^{-1}v_t + \frac12 A(u) = 0, \quad t>0, \quad v(0) = h(u_0)^{1/2},
\end{equation}
which is formally equivalent to \eqref{1.eq}.
The semidiscrete scheme reads as
\begin{equation}\label{1.scheme}
  h(w_k)^{1/2}h'(w_k)^{-1}\rho(E)v_{k} + \frac{\tau}{2} A(w_k) = 0, \quad
  w_k = h^{-1}\big((\sigma(E)v_{k})^2\big), \quad k\ge 0.
\end{equation}
Note that $v_k$ approximates $h(u(t_k))^{1/2}$ and $w_k$ approximates $u(t_k)$.

The {\em first aim} of this paper is to prove that this scheme dissipates the discrete
entropy $H[V_k]$ with $V_k=(v_k,\ldots,v_{k+p-1})$ 
(see Proposition \ref{prop.g}), i.e.,
\begin{equation}\label{1.ent.diss}
  H[V_{k+1}] - H[V_k] \le 0,
\end{equation}
which is the discrete analogue of \eqref{1.ent}. Here, $H[V_k]$ is defined by
$$
  H[V_k] = \frac12\sum_{i,j=0}^{p-1}G_{ij}(v_{k+i},v_{k+j}),
$$
and $G=(G_{ij})$ is the matrix occuring in the definition of the G-stability
(see Section \ref{sec.g} for details). Note that 
$H[V_k]\ge 0$ since $G$ is assumed to be positive definite. 

The {\em second aim} of this paper is to prove that scheme \eqref{1.scheme} possesses
an (entropy-dissipa\-ting) nonnegative weak solution. 
For the existence proof we need additional
assumptions on the operator $A$. The main conditions are that equation \eqref{1.eq}
is nonnegativity-preserving and possesses two entropies,
$h(u)=u^\alpha$ for $1<\alpha<2$ and $h(u)=u\log u$ (assuming that these
expressions are defined). The nonnegativity-preservation of our scheme is
inherited by the first condition.
This property is proved by using the entropy density $h(u)=u\log u$ and
the variable transformation $u=e^y>0$. The entropy density $h(u)=u^\alpha$
allows us to show the entropy-dissipation of our scheme. We believe that the
condition $1<\alpha<2$ is technical. It is needed to control the discrete
time derivative in \eqref{1.scheme} when using the test function $\log w_k$
(see Lemma \ref{lem.est.t}).

A general existence proof would require more assumptions on $A$ which might
restrict the applicability of our results. Therefore, we prefer to demonstrate the
flexibility of our ideas by investigating two very different examples for $A$.
The first example is a cross-diffusion system, the second
one is a highly nonlinear equation of fourth order. In particular, our scheme
is {\em not} restricted to scalar or second-order diffusion equations.
For these examples, which are detailed in Sections \ref{sec.1} and \ref{sec.2}, 
we give rigorous proofs of the existence of semidiscrete weak
solutions to scheme \eqref{1.scheme}. 

Our {\em third aim} is to prove the second-order convergence of the
one-leg scheme \eqref{1.scheme}. Given a sequence of positive solutions $(v_k)$ 
to \eqref{1.scheme} and a smooth positive solution $v=u^{\alpha/2}$ to \eqref{1.eqv},
this means that there exists $C>0$ such that for all sufficiently small $\tau>0$,
\begin{equation}\label{1.tau2}
  \|v_k-v(t_k)\| \le C\tau^2, \quad t_k=\tau k,\ k\ge 1.
\end{equation}
For this result, we assume that the mapping $v\mapsto v^{1-\alpha/2}A(v^{2/\alpha})$
satisfies a one-sided Lip\-schitz condition and that the scheme $(\rho,\sigma)$
is of second order (see Theorem \ref{thm.second} for details). For instance,
if $A$ is the fourth-order operator of the second example and $\alpha=1$,
this assumption is satisfied \cite{JuPi03}. The proof of \eqref{1.tau2}
is similar to the proof of \cite[Theorem V.6.10]{HaWa91}, based on an idea
of Hundsdorfer and Steininger \cite{HuSt91}. For convenience, we present the
full proof, specialized to the present situation in general Hilbert spaces.

The paper is organized as follows.
In Section \ref{sec.main}, we state the main results on the existence of
semidiscrete weak solutions and the convergence rate.
General one-leg multistep schemes, which dissipate the energy or entropy, 
are discussed in Section \ref{sec.g}.
Section \ref{sec.ex} is devoted to the existence analysis. We detail
the strategy of the existence proof in a general context and prove Theorems
\ref{thm.skt} and \ref{thm.dlss}. The second-order convergence
of the scheme is proved in Section \ref{sec.second}. 
Numerical examples in Section \ref{sec.numerics}
for the cross-diffusion population model,
using the two-step BDF and the so-called $\gamma$-method (see Remark \ref{rem.ex}), 
show that the discrete entropy $H[V_k]$ decays exponentially fast to the 
stationary state. Finally, in the Appendix, we derive
the family of all G-stable second-order one-leg schemes.


\section{Main results}\label{sec.main}

We state the existence theorems for the semidiscretized
cross-diffusion system and fourth-order
equation and a theorem on the second-order convergence of the one-leg scheme.

\subsection{Cross-diffusion population system}\label{sec.1}

The first example is the cross-diffusion population model of Shigesada,
Kawasaki, and Teramoto \cite{SKT79}:
\begin{align}
  u^{(1)}_t - \diver\big((d_1+a_1u^{(1)}+u^{(2)})\na u^{(1)}
  +u^{(1)}\na u^{(2)}\big) &= 0, \label{1.skt1} \\
  u^{(2)}_t - \diver\big((d_2+a_2u^{(2)}+u^{(1)})\na u^{(2)}
  +u^{(2)}\na u^{(1)}\big) &= 0, \quad t>0, \label{1.skt2}
\end{align}
where $u^{(j)}(x,t)$ is the density of the $j$-th species,
$d_1$, $d_2>0$ are the diffusion coefficients, $a_1$, $a_2>0$
denote the self-diffusion coefficients,
and the expression $\na(u^{(1)}u^{(2)})=u^{(1)}\na u^{(2)}+u^{(2)}\na u^{(1)}$ is the 
cross-diffusion term.
The above system has been scaled in such a way that the coefficient of
the cross-diffusion term is equal to one (see \cite{GGJ03} for details).
The equations are solved on the $d$-dimensional torus $\T^d$ with the initial conditions
\begin{equation}\label{1.skt3}
  u^{(1)}(0) = u^{(1)}_0, \quad u^{(2)}(0) = u^{(2)}_0\quad\mbox{in }\T^d.
\end{equation}
Our results are also valid for homogeneous Neumann boundary conditions
and suitable reaction (Lotka-Volterra) terms.
This model describes the time evolution of two competing species neglecting
Lotka-Volterra terms and effects of the environment. 
The basic idea is that the primary cause of dispersal is migration to avoid crowding
instead of just random motion, modeled by diffusion.
The model can be derived from
a random walk on a lattice by assuming that the transition probabilities
for a one-step jump depend linearly on the species' numbers.

The main feature of system \eqref{1.skt1}-\eqref{1.skt2} is that its diffusion
matrix is generally neither symmetric nor positive definite. 
Using entropy methods, the implicit Euler time discretization 
and a partial finite-difference
approximation, the global existence of weak solutions was shown
in \cite{ChJu04}. Instead of discretizing the cross-diffusion term by finite 
differences, an elliptic regularization was employed in \cite{ChJu06}.
Another (simpler) regularization was suggested in the finite-element context
by Barrett and Blowey \cite{BaBl04} by using an approximate entropy functional.
For the one-dimensional equations, a temporally semi-discrete approximation
was investigated in \cite{GGJ03}.  
Andreianov et al.\ \cite{ABB11} employed a finite-volume method assuming
positive definite diffusion matrices. A deterministic particle method with a
Peaceman-Rachford operator splitting in time was developed by Gambino et al.\
\cite{GLS09}. In all these approaches (except \cite{GLS09}), an implicit
Euler discretization was used. We allow for (G-stable) higher-order 
time discretizations.

We choose the entropy density $h(u)=(u^{(1)})^\alpha + (u^{(2)})^\alpha$ 
for $\alpha>1$ and $u=(u^{(1)},u^{(2)})$ and the discrete entropy
\begin{equation}\label{1.H}
  H[V_k] = \frac12\sum_{i,j=0}^{p-1}G_{ij}\int_{\T^d}v_{k+i}\cdot v_{k+j}dx
  = \frac12\sum_{i,j=0}^{p-1}G_{ij}\int_{\T^d}(v_{k+i}^{(1)}v_{k+j}^{(1)}
  + v_{k+i}^{(2)}v_{k+j}^{(2)})dx.
\end{equation}

\begin{theorem}[Semidiscrete population system]\label{thm.skt}
Let $d\le 3$, $1<\alpha<2$, and $4a_1a_2\ge\max\{a_1,a_2\}+1$.
Let $v_0,\ldots,v_{p-1}\in L^2(\T^d)^2$ be nonnegative
componentwise. Furthermore, let the scheme $(\rho,\sigma)$, defined in 
\eqref{1.rho}, be G-stable (hence, $p\le 2$) and assume that $\alpha_p>0$
and $\beta_p>0$. Then there exists a sequence of weak solutions $(v_k,w_k)
=(v_{k}^{(1)},v_{k}^{(2)},w_{k}^{(1)},w_{k}^{(2)})\in L^2(\T^d)^2
\times W^{1,3/2}(\T^d)^2$ to 
\begin{align}
  \frac{2}{\alpha\tau}(w_{k}^{(1)})^{1-\alpha/2}\rho(E)v_{k}^{(1)} + \diver\big(
  (d_1+a_1 w_{k}^{(1)}+w_{k}^{(2)})\na w_{k}^{(1)} 
  + w_{k}^{(1)}\na w_{k}^{(2)}\big) &= 0, \label{1.d1} \\
  \frac{2}{\alpha\tau}(w_{k}^{(2)})^{1-\alpha/2}\rho(E)v_{k}^{(2)} + \diver\big(
  (d_2+a_2 w_{k}^{(2)}+w_{k}^{(1)})\na w_{k}^{(2)} 
  + w_{k}^{(2)}\na w_{k}^{(1)}\big) &= 0, \label{1.d2}
\end{align}
where $w_{k}^{(j)}=(\sigma(E)v_{k}^{(j)})^{2/\alpha}$, 
$\sigma(E)v_{k}^{(j)}\ge 0$, and 
$w_{k}^{(j)}\in L^{3\alpha}(\T^d)^2$ $(j=1,2$, $k\ge 0)$. 
The scheme dissipates the entropy in the sense
\begin{equation}\label{1.skt.ent}
  H[V_{k+1}] + \frac{2\tau}{\alpha^2}(\alpha-1)\int_{\T^d}\big(d_1
  |\na w_{1,k}^{\alpha/2}|^2 + d_2|\na w_{2,k}^{\alpha/2}|^2\big)dx
  \le H[V_k],
\end{equation}
where $H[V_k]$ is defined in \eqref{1.H}.
\end{theorem}

The condition $4a_1a_2\ge\max\{a_1,a_2\}+1$ is needed to prove that
the cross-diffusion system \eqref{1.skt1}-\eqref{1.skt2} dissipates
the entropy for $1<\alpha<2$; see Lemma \ref{lem.assA}.


\subsection{Fourth-order quantum diffusion equation}\label{sec.2}

We consider the quantum diffusion equation, also called
Derrida-Lebowitz-Speer-Spohn (DLSS) equation,
\begin{equation}\label{1.dlss}
  u_t + \na^2:(u\na^2\log u) = 0, \quad t>0, \quad u(0)=u_0 \quad\mbox{in }\T^d,
\end{equation}
where $\na^2 u$ is the Hessian matrix of $u$ and $A:B=\sum_{i,j}A_{ij}B_{ij}$
is the Frobenius inner product between matrices. 
This equation is the zero-temperature and zero-field limit of the
quantum drift-diffusion model, which describes the evolution of the electron
density $u(t)$ in a quantum semiconductor device \cite{Jue09}. It was
derived in \cite{DMR05} from a relaxation-time Wigner equation. Its one-dimensional
version was derived in \cite{DLSS91} in a suitable scaling limit from the time-discrete
Toom model, where $u$ is related to a random variable.

The global-in-time existence of nonnegative weak solutions to \eqref{1.dlss}
was proven in \cite{GST09,JuMa08}. 
Most of the numerical schemes proposed for \eqref{1.dlss}
are based on an implicit Euler discretization in one space dimension. 
In \cite{JuPi01}, the convergence of a positivity-preserving semidiscrete 
Euler scheme was shown.
A fully discrete finite-difference scheme which preserves the positivity, 
mass, and physical entropy was derived in \cite{CJT03}. 
D\"uring et al.\ \cite{DMM10} employed the variational structure of \eqref{1.dlss}
on a fully discrete level and introduced a discrete minimizing movement scheme.
Finally, a two-step BDF method was applied to
\eqref{1.dlss} in \cite{BEJ12} and the second-order convergence of semidiscrete
solutions was shown. Here, we generalize \cite{BEJ12} by allowing for general
(G-stable) second-order time discretizations.

We choose the entropy density $h(u)=u^\alpha$ 
for $\alpha>1$ and the discrete entropy
$H[V_k]=\frac12\sum_{i,j=0}^{p-1}G_{ij}\int_{\T^d}v_{k+i}v_{k+j}dx$.

\begin{theorem}[Semidiscrete DLSS equation]\label{thm.dlss}
Let $1\le d\le 3$, $1<\alpha<(\sqrt{d}+1)^2/(d+2)$, and let 
$v_0,\ldots,v_{p-1}\in L^2(\T^d)$ be nonnegative. Furthermore, let the scheme 
$(\rho,\sigma)$ be G-stable (hence, $p\le 2$)
and assume that $\alpha_p>0$ and $\beta_p>0$ hold.
Then there exists a sequence of weak solutions $(v_k,w_k)\in 
L^2(\T^d)\times L^\infty(\T^d)$ to 
$$
  \frac{2}{\alpha\tau}w_k^{1-\alpha/2}\rho(E)v_{k} + \na^2:(w_k\na^2\log w_k) = 0,
$$
satisfying $w_k=(\sigma(E)v_{k})^{2/\alpha}$, $\sigma(E)v_{k}\ge 0$, 
$w_k^{\alpha/2}\in H^2(\T^d)$, and 
$w_{k}^{1/2}\in W^{1,2\alpha}(\T^d)$, in the following sense: 
For all $\phi\in W^{2,\alpha/(\alpha-1)}(\T^d)$ and $k\ge 0$,
\begin{equation}\label{ex.weak0}
  \frac{1}{\tau}\int_{\T^d}w_{k}^{1-\alpha/2}\rho(E)v_{k} \phi dx
  + \int_{\T^d}\big(w_k^{1-\alpha/2}\na^2 w_{k}^{\alpha/2} 
  - \alpha^2\na w_{k}^{1/2}\otimes\na w_{k}^{1/2}\big):\na^2\phi dx = 0.
\end{equation}
The scheme dissipates the entropy in the sense
\begin{equation}\label{ex.H}
  H[V_{k+1}] + \frac{\alpha}{4}\kappa_\alpha\tau
  \int_{\T^d}(\Delta w_{k}^{\alpha/2})^2 dx \le H[V_k],
\end{equation}
where $\kappa_\alpha>0$ only depends on $\alpha$ and $d$.
\end{theorem}

Again, the condition $1<\alpha<(\sqrt{d}+1)^2/(d+2)$ is needed to derive
the entropy dissipation of \eqref{1.dlss}; see \cite{JuMa08}.


\subsection{Second-order convergence rate}\label{sec.thm3}

Let $\V\hookrightarrow \HH\hookrightarrow \V'$ be a Gelfand tripel 
\cite[Section 23.4]{Zei90}, where $\V$ is a Banach space and $\HH$ is
a Hilbert space with scalar product $(\cdot,\cdot)$ and norm $\|\cdot\|$. 
We assume that we can define the notion of positivity on $\HH$.
Furthermore, let $A:D(A)\to \V'$ be a (nonlinear) operator with domain $D(A)\subset\V$
and let $h(u)=u^\alpha$ for
$u\in D(A)$, $u>0$ with $\alpha\ge 1$. (We assume that the expression $u^\alpha$
makes sense in $\V$.) Then, given $v_0=u_0^{\alpha/2}$, let $v_1$ be the solution
to the implicit Euler scheme (which is assumed to exist)
\begin{equation}\label{euler}
  \frac{2}{\alpha\tau}(v_1-v_0) + v_1^{1-2/\alpha}A(v_1^{2/\alpha}) = 0.
\end{equation}
We assume that the scheme $(\rho,\sigma)$ with $p=2$ is G-stable and that
the differentation error $\delta_D(t)$ and the interpolation error $\delta_I(t)$,
defined by \cite[Section V.6]{HaWa91}
\begin{equation}\label{deltaD}
  \delta_D(t) = \rho(E)v(t) - \tau v_t(t+2\tau), \quad
	\delta_I(t) = \sigma(E)v(t) - v(t+2\tau)
\end{equation}
are of second order (see Section \ref{sec.g.ene}).

\begin{theorem}\label{thm.second}
Let $(v_k)$ be a sequence of smooth solutions to \eqref{1.scheme} and \eqref{euler}
satisfying $\sigma(E)v_k>0$ and let $u$ be a smooth positive solution to \eqref{1.eq}.
Let the above assumptions on the scheme $(\rho,\sigma)$ hold.
We assume that the mapping 
$v\mapsto B(v)=\frac{\alpha}{2} v^{1-2/\alpha}A(v^{2/\alpha})$ is
well defined and satisfies the one-sided Lipschitz condition
$$
  \langle B(v)-B(\bar v),v-\bar v\rangle \ge -\kappa_1\|v-\bar v\|^2
	\quad\mbox{for all }v^{2/\alpha},\bar v^{2/\alpha}\in D(A)
$$
for some $\kappa_1>0$.
Then there exist $\tau_0>0$ and $C>0$ such that for all $0<\tau\le \tau_0$,
$$
  \|v_k-u(t_k)^{\alpha/2}\| \le C\tau^2, \quad t_k = \tau k,\ k\ge 0.
$$
\end{theorem}

The one-sided Lipschitz condition is also needed in \cite[Section V.6]{HaWa91}.
It is satisfied, for instance, 
for the operator of the population model \eqref{1.skt1}-\eqref{1.skt2}
with domain contained in $W^{1,\infty}(\Omega)$, or for monotone operators $B$.
We give some examples for the latter case.
If $A:D(A)\to\V'$ is any monotone operator, the assumption of the theorem
is trivially satisfied for $\alpha=2$ since then $A=B$. 
In this situation, we recover the ``energy'' method described in the introduction.

Next, let $A:D\to H^{-2}(\Omega)$ with 
$A(w)=\na^2:(w\na^2\log w)$ for $w\in D=\{w\in H^2(\Omega):w>0$ in $\Omega\}$ 
and $\Omega\subset\R^d$ ($d\le 3$)
be the operator of the DLSS equation. It is shown in \cite[Lemma 3.5]{JuPi03}
that $v\mapsto v^{-1}A(v^2)$ is monotone in the sense of 
$\langle v_1^{-1}A(v_1^2)-v_2^{-1}A(v_2^2),v_1-v_2\rangle \ge 0$ for
$v_1^2$, $v_2^2\in D$. This operator satisfies the assumptions of Theorem 
\ref{thm.second} for $\alpha=1$.
In fact, the above theorem is a generalization of Theorem 2 in \cite{BEJ12},
which is proved for $\alpha=1$ and the two-step BDF method only.
Another example is $\alpha=4/3$ and the fast-diffusion
operator $A(u)=-\Delta(u^{1/3})$ with Dirichlet boundary conditions,
although we do not study this operator here.


\section{General one-leg multistep schemes}\label{sec.g}

We wish to semi-discretize the Cauchy problem \eqref{1.eq} in its weak formulation
\begin{equation}\label{g.u}
  \langle u_t(t),\phi\rangle + \langle A(u(t)),\phi\rangle = 0 
  \quad\mbox{for all }\phi\in \V,\ t>0, \quad u(0)=u_0,
\end{equation}
where $\langle\cdot,\cdot\rangle$ is the dual product between $\V'$ and $\V$
(see Section \ref{sec.thm3} for the notations).
We assume that there exists a (smooth) solution $u:[0,T]\to D(A)$ to \eqref{g.u}.


\subsection{One-leg schemes and energy dissipation}\label{sec.g.ene}

We recall some basic notions of one-leg schemes and G-stability.
We introduce the time steps $t_k=\tau k$, where $\tau>0$ is the time step size
and $k\in\N$. One-leg methods can be formulated in compact form by introducing
the polynomials
$$
  \rho(\xi) = \sum_{j=0}^p \alpha_j \xi^j, \quad
  \sigma(\xi) = \sum_{j=0}^p \beta_j \xi^j,
$$
where $\alpha_j$, $\beta_j\in\R$, $\alpha_p\neq 0$, and we normalize $\sigma(1)=1$.
Let $Eu_k=u_{k+1}$ be the forward time shift, defined on the sequence $(u_k)$.
Then
$$
  \rho(E)u_k = \sum_{j=0}^p \alpha_j u_{k+j}, \quad
  \sigma(E)u_k = \sum_{j=0}^p \beta_j u_{k+j}, \quad k\ge 0.
$$
As mentioned in the introduction, the standard one-leg discretization of
\eqref{1.eq} reads as
\begin{equation}\label{g.uk}
  \tau^{-1}\rho(E)u_k + A(\sigma(E)u_k) = 0, \quad k\ge 0,
\end{equation}
where $u_k$ and $\sigma(E)u_k$ approximate $u(t_k)$ and $\tau^{-1}\rho(E)u_k$
approximates $u_t(t_k)$. The values $u_0,\ldots,u_{p-1}$ are assumed to be given.

According to \cite[Exercise 1a, Section V.6]{HaWa91}, the conditions
$\rho(1)=0$, $\rho'(1)=\sigma(1)=1$ imply the consistency of the 
scheme $(\rho,\sigma)$. If additionally $\rho'(1)+\rho''(1)=2\sigma'(1)$,
the scheme is second-order accurate, i.e., the differentiation error
\eqref{deltaD} satisfies $\|\delta_D(t)\|\le C_D\tau^3$ uniformly in $t\in(0,T)$.
The constant $C_D>0$ depends on the $L^\infty(0,T;\HH)$ norm of $v_{ttt}$. 
Furthermore, if $\sigma(1)=1$ and $\sigma'(1)=2$, the interpolation error
is of second order, i.e.\ $\|\delta_I(t)\|\le C_I\tau^2$ uniformly in $t\in(0,T)$,
and $C_I>0$ depends on $v_{tt}$.

Dahlquist \cite{Dah63} has proven that any A-stable scheme 
$(\rho,\sigma)$ is at most of second order. He related the discrete energy
dissipation to a stability condition, called G-stability. We say that
$(\rho,\sigma)$ is G-stable \cite[Def.~2.4]{Hil97}
if there exists a symmetric, positive definite matrix $G=(G_{ij})\in\R^{p\times p}$ 
(called G-matrix in the following)
such that for any sequence $(u_k)$ defined on $\V$,
\begin{equation}\label{g.stable}
  (\rho(E)u_k,\sigma(E)u_k) \ge \frac12\|U_{k+1}\|_G^2 - \frac12\|U_k\|_G^2
  \quad\mbox{for all }k\in\N,
\end{equation}
where the G-norm is given by
$$
  \|U_k\|_G^2 = \sum_{i,j=0}^{p-1} G_{ij}(u_{k+i},u_{k+j}), \quad
  U_k = (u_k,\ldots,u_{k+p-1}).
$$
Any scheme $(\rho,\sigma)$, for which $\rho(\xi)$ and $\sigma(\xi)$
are coprime polynomials, is G-stable if and only if it is A-stable \cite{BaCr89,Dah78}.
The proof in \cite{BaCr89} provides constructive formulas for the matrix $G$
(also see \cite[Section V.6]{HaWa91}). The
G-stability and the monotonicity condition \eqref{1.mono} imply energy dissipation 
since, formally, by \eqref{g.uk},
$$
  \frac12\|U_k\|_G^2 - \frac12\|U_{k-1}\|_G^2
  \le (\rho(E)u_k,\sigma(E)u_k) = -\tau(A(\sigma(E)u_k),\sigma(E)u_k) \le 0.
$$
In particular, the discrete energy $k\mapsto \frac12\|U_k\|_G^2$ is nonincreasing.
In the appendix, we derive all second-order one-leg schemes which are G-stable.

\begin{remark}\rm\label{rem.ex}
We give some known examples of G-stable one-leg methods.
Examples (ii) and (iii) are included in the family of schemes derived in the appendix.
\begin{enumerate}[{\rm (i)}]
\item
The (first-order) {\em implicit mid-point rule} is defined by $p=1$, 
$(\alpha_0,\alpha_1)=(-1,1)$ and $(\beta_0,\beta_1) = (\frac12,\frac12)$. 
Then the G-norm coincides with the norm on $\HH$.
\item
The {\em two-step BDF method} is defined by $p=2$, $(\alpha_0,\alpha_1,\alpha_2)
=(\frac12,-2,\frac32)$ and $(\beta_0,\beta_1,\beta_2)$ $=(0,0,1)$.
It is of second order and its G-matrix equals
$$
  G = \frac12\begin{pmatrix} 1 & -2 \\ -2 & 5 \end{pmatrix}.
$$
\item
A family of two-step one-leg methods is proposed in 
\cite{DLN83,KuSh05} with $p=2$ and
\begin{align*}
  (\alpha_0,\alpha_1,\alpha_2) &= \frac{1}{\gamma+1}(-\gamma,\gamma-1,1), \\
  (\beta_0,\beta_1,\beta_2) 
  &= \frac{1}{2(\gamma+1)^2}\big(\gamma(\gamma+3),(\gamma-1)^2,3\gamma+1\big),
\end{align*}
where $0<\gamma\le 1$ is a free parameter. 
In \cite{KuSh05}, the value
$\gamma=9-4\sqrt{5}\approx 0.055$ is derived by optimizing the stability at
infinity for this method, whereas the authors of \cite{DLN83} minimize the
error constant of the method, which leads to $\gamma=1/5$. 
The scheme is of second order; the G-stability follows from the identity
$$
  (\rho(E)u,\sigma(E)u)
  = U_1^\top GU_1 - U_0^\top GU_0
  + \frac{1-\gamma}{2(\gamma+1)^3}(u_0-2u_1+u_2)^2
$$
for all $u=(u_0,u_1,u_2)^\top\in\R^3$, where $U_1=(u_1,u_2)^\top$, 
$U_0=(u_0,u_1)^\top$, and the G-matrix
$$
  G = \frac{1}{2(\gamma+1)}\begin{pmatrix} \gamma & 0 \\ 0 & 1 \end{pmatrix}
$$
is diagonal.
\qed
\end{enumerate}
\end{remark}


\subsection{One-leg schemes and entropy dissipation}\label{sec.g.ent}

In this subsection, we introduce general one-leg schemes which dissipate the
entropy. To this end, let $h:D(A)\to \V$ be a differentiable and invertible function.
We assume that we can define the notion of nonnegativity on $\V$ and that
$h(u)\ge 0$ for all $u\in D(A)$, $u\ge 0$.
Our main hypothesis is 
\begin{equation}\label{g.assA}
  \langle A(\phi),h'(\phi)\rangle \ge 0 \quad\mbox{for all }\phi\in D(A).
\end{equation}
We have shown in \eqref{1.ent} that this assumption implies 
that $H[u]$ is a Lyapunov functional for \eqref{1.eq},
where we identify the function $h'(u)$ with the Fr\'echet derivative $H'[u]$.
Instead of discretizing \eqref{g.u} directly, we consider the weak formulation
\begin{equation}\label{g.u2}
  2(v_t,h(u)^{1/2}h'(u)^{-1}\phi) + \langle A(u),\phi\rangle = 0, \quad
  t>0, \quad v(0) = h(u_0)^{1/2},
\end{equation}
where $v=h(u)^{1/2}$, which is formally equivalent to \eqref{g.u}. 
In order to be well defined, we assume that the product 
$h(u)^{1/2}h'(u)^{-1}\phi$ is an element of $\HH$ and that $u\ge 0$.

With the notations from Section \ref{sec.g.ene}, let $(\rho,\sigma)$ be 
a consistent scheme. Furthermore, let $u_k$ approximate $u(t_k)$ and define
$$
  v_k = h(u_k)^{1/2}, \quad w_k = h^{-1}\big((\sigma(E)v_k)^2\big), \quad k\in\N,
$$
supposing that $(\sigma(E)v_k)^2\in\V$. Then $v_k$ approximates $h(u(t_k))^{1/2}$
and $w_k$ is an approximation of $u(t_k)$.
The one-leg multistep approximation of \eqref{g.u2} is defined by
\begin{equation}\label{g.disc}
  \frac{2}{\tau}(\rho(E)v_{k},h(w_k)^{1/2}h'(w_k)^{-1}\phi) 
  + \langle A(w_k),\phi\rangle = 0, 
  \quad w_k = h^{-1}\big((\sigma(E)v_{k})^2\big),
\end{equation}
where $k\ge 0$ and the values $v_0,\ldots,v_{p-1}$ are given. The existence of weak
solutions is investigated in Section \ref{sec.ex}.
The following proposition states that this scheme dissipates the discrete entropy 
$$
  H[V_k] = \frac12\|V_k\|_G^2 = \frac12\sum_{i,j=0}^{p-1} G_{ij}(v_{k+i},v_{k+j}), 
  \quad V_k = (v_k,\ldots,v_{k+p-1}),
$$
if the scheme is G-stable. 

\begin{proposition}\label{prop.g}
Let $(\rho,\sigma)$ be a G-stable scheme and assume that \eqref{g.assA} holds.
Let $(v_k)$ be a sequence of solutions to \eqref{g.disc} such that $h'(w_k)\in\V$. 
Then the scheme dissipates the discrete entropy, 
i.e., $k\mapsto H[V_k]$ is nonincreasing, where $V_k=(v_k,\ldots,v_{k+p-1})$.
\end{proposition}

\begin{proof}
We employ the definition $h(w_k)^{1/2}=\sigma(E)v_{k}$ and the 
test function $\phi=h'(w_k)\in\V$ in \eqref{g.disc}:
$$
  (\rho(E)v_{k},\sigma(E)v_{k}) = 
  (\rho(E)v_{k},h(w_k)^{1/2}) = -\frac{\tau}{2}\langle A(w_k),h'(w_k)\rangle
  \le 0.
$$
Because of the G-stability \eqref{g.stable},
we find that for $k\ge 0$,
$$
  H[V_{k+1}] - H[V_{k}] = \frac12\|V_{k+1}\|_G^2 - \frac12\|V_{k}\|_G^2
  \le (\rho(E)v_k,\sigma(E)v_k) \le 0,
$$
finishing the proof.
\end{proof}

\begin{remark}\rm
Let $1\in\HH$, $(u_k,1)=(u_0,1)$ for all $k=0,\ldots,p-1$, and
$\langle A(\sigma(E)u_k),1\rangle=0$. Then scheme \eqref{g.uk} preserve the
mass, i.e., $(u_k,1)=(u_0,1)$ for all $k\in\N$. 
This is generally not the case for scheme \eqref{g.disc}.
However, choosing $\phi=1$ in \eqref{g.disc}, it follows that
$(\rho(E)v_k,h(w_k)^{1/2}h'(w_k)^{-1})=0$, which approximates
$\pa_t(u(t),1) = 0$.
\end{remark}


\section{Existence of semi-discrete solutions}\label{sec.ex}

In this section, we first sketch the existence proof in the general
situation \eqref{g.disc} and explain the main ideas. Rigorous proofs 
are presented for the cross-diffusion system \eqref{1.skt1}-\eqref{1.skt2}
in Section \ref{sec.skt} and for the fourth-order quantum diffusion equation
\eqref{1.dlss} in Section \ref{sec.dlss}.

\subsection{General strategy}\label{sec.gen}

The proof of the existence of weak solutions
to scheme \eqref{g.disc} is based on a regularization procedure and a
fixed-point argument. In order to pass to the limit of vanishing regularization 
parameters, some compactness is needed, which strongly depends on the properties
of the operator $A$ and the choice of the function spaces. 
Therefore, we only detail those parts of the proof which concern
the time discretization and refer to the subsequent subsections for full proofs
with specific examples of $A$ and the function spaces.

Let the assumptions at the beginning of Sections \ref{sec.thm3} and
\ref{sec.g} hold and let
$h(u)=u^\alpha$ with $1<\alpha<2$. 
Furthermore, let $\HH=L^2(\Omega)$, where $\Omega\subset\R^d$ is some bounded domain.
Let $D(A)\subset\V$ be a subspace such that $u^{\alpha-1}$, $\log u\in\V$ for
all $u\in D(A)$ (examples are given in the following subsections). We suppose:
\renewcommand{\theenumi}{\roman{enumi}} 
\renewcommand{\labelenumi}{(\theenumi)}
\newcounter{saveenumi}
\begin{enumerate}
\item The operator $A:D(A)\to\V'$ satisfies
$\langle A(u),u^{\alpha-1}\rangle\ge 0$ and
$\langle A(u),\log u\rangle\ge 0$ for all $u\in D(A)$.
\item There exists a linear bounded operator 
$\widetilde A[z]:\V\to\V'$ (``linearization'' of $A$ around $z\in D(A)$) 
such that $\widetilde A[e^y](y)=A(e^y)$ for all $e^y\in D(A)$ 
and $\langle \widetilde A[e^z](y),y\rangle\ge 0$ for all $e^z\in D(A)$
and $y\in\V$.
\end{enumerate}
\setcounter{saveenumi}{\value{enumi}}
The first condition in (i) corresponds to \eqref{g.assA}; the second condition
is needed to derive uniform estimates in the variable $\log u$
for the fixed-point argument.
Assumption (ii) is used to apply the Lax-Milgram lemma.

{\em Step 1: Exponential variable transformation.}
First, we formulate scheme \eqref{g.disc} in the variable $y=\log w_k$.
This transformation provides the positivity of the numerical solution
$w_k=e^y$. Let $k\in\N_0$ and $v_k,\ldots,v_{k+p-1}\in\HH$ be given.
Let $\Z\subset\V$ be a Banach space satisfying $\Z\subset L^\infty(\Omega)$ and
$e^{\beta y}\in D(A)\cap \Z$ for all $y\in\Z$ and $\beta>0$.
Defining $\delta_j = \alpha_j - \alpha_p\beta_j/\beta_p$, we wish to solve
the problem in the variable $y$,
\begin{equation}\label{ex.y}
  \frac{2}{\alpha\tau}e^{(1-\alpha/2)y}\left(\frac{\alpha_p}{\beta_p}e^{\alpha y/2} 
  + \sum_{j=0}^{p-1}\delta_j v_{k+j}\right)
  + A(e^y) + \eps L(y) = 0\quad\mbox{in }\Z'.
\end{equation}
The relation to the original problem will be shown in \eqref{ex.disc} below.
We have added a regularization operator $L:\Z\to\Z'$, 
which is needed to derive uniform estimates in terms of $y$.
The operator $L$ is supposed to satisfy the following conditions:
\begin{enumerate}
\setcounter{enumi}{\value{saveenumi}}
\item There exist $C>0$ and $\kappa_1>0$ such that
$\langle L(y),e^{(\alpha-1)y}\rangle\ge -C$ and 
$\langle L(y),y\rangle\ge \kappa_1\|y\|_\Z^2$ for all $y\in\Z$.
\item There exists a linear bounded operator $\widetilde L[z]:\Z\to\Z'$ 
(``linearization'' of $L$ around $z\in\Z$) 
such that $\widetilde L[y](y)=L(y)$ for all $y\in\Z$
and $\langle \widetilde L[z](y),y\rangle\ge \kappa_0\|y\|_\Z^2$ 
for all $e^z\in D(A)$ and $y\in\Z$, where $\kappa_0>0$.
\end{enumerate}
The first condition in (iii) is important to derive the discrete entropy estimates;
the second condition ensures the coercivity of a suitable  bilinear form
in the Lax-Milgram argument.
Again, Assumption (iv) is employed in the ``linearization'' of the problem.

We claim that any solution $y\in\Z$ to \eqref{ex.y}
defines a solution to 
\begin{equation}\label{ex.disc}
  \frac{2}{\tau}\int_{\Omega}h(w_k)^{1/2}h'(w_k)^{-1}\rho(E)v_k\phi dx
  + \langle A(w_k),\phi\rangle + \eps\langle L(y),\phi\rangle = 0
\end{equation}
for all $\phi\in\Z$, where $w_k=e^y\in D(A)$ and 
$h(w_k)^{1/2}h'(w_k)^{-1}=(1/\alpha)w_k^{1-\alpha/2}$.
Indeed, we define 
\begin{equation}\label{ex.v}
  v_{k+p} := \frac{1}{\beta_p}w_{k}^{\alpha/2} 
  - \sum_{j=0}^{p-1}\frac{\beta_j}{\beta_p} v_{k+j}\in \HH.
\end{equation}
Then we insert $v_{k+p}$ in the definition \eqref{1.rho} of $\sigma(E)v_k$, 
leading to
$$
  \sigma(E) v_{k} = \sum_{k=0}^{p-1}\beta_j v_{k+j} + \beta_p v_{k+p} 
	= w_{k}^{\alpha/2}.
$$
Since $w_{k}=e^y>0$, we infer that $\sigma(E)v_{k}>0$. 
Inserting \eqref{ex.v} in the definition \eqref{1.rho} of $\rho(E)v_{k}$, we find that
\begin{align}
  \rho(E) v_{k} 
  &= \alpha_p v_{k+p} + \sum_{j=0}^{p-1} \alpha_j v_{k+j}
  = \frac{\alpha_p}{\beta_p}w_{k}^{\alpha/2}
  + \sum_{j=0}^{p-1}\left(\alpha_j-\frac{\alpha_p}{\beta_p}\beta_j\right) v_{k+j} 
  \nonumber \\
  &= \frac{\alpha_p}{\beta_p}e^{\alpha y/2} + \sum_{j=0}^{p-1}\delta_j v_{k+j},
  \label{ex.rhov}
\end{align}
employing the definition $\delta_j = \alpha_j-\alpha_p\beta_j/\beta_p$.
Replacing the brackets in \eqref{ex.y} by $\rho(E) v_{k}$, 
we conclude that $(v_{k+p},w_{k})$ solves \eqref{ex.disc}.

{\em Step 2: Definition of the fixed-point operator.} 
Let $\X$ be a Banach space such that the embedding $\Z\hookrightarrow \X$
is compact and $\X\subset L^\infty(\Omega)$. 
Let $z\in\X$, $\eta\in[0,1]$ be given. We define on $\Z$ the linear forms
\begin{align*}
  a(y,\phi) &= \langle \widetilde A[e^z](y),\phi\rangle
  + \eps \langle\widetilde L[z](y),\phi\rangle, \\
  F(\phi) &= -\frac{2\eta}{\alpha\tau}\int_{\Omega}e^{(1-\alpha/2)z}
  \left(\frac{\alpha_p}{\beta_p}e^{\alpha z/2} 
  + \sum_{j=0}^{p-1}\delta_j v_{k+j}\right)
  \phi dx.
\end{align*}
Note that $e^{\gamma z}$ is well defined for all $\gamma>0$
since $z\in\Z\subset L^\infty(\Omega)$. We wish to find $y\in\Z$ such that
\begin{equation}\label{ex.LM}
  a(y,\phi) = F(\phi) \quad\mbox{for all }\phi\in \Z.
\end{equation}
By Assumptions (ii) and (iv), the forms $a$ and $F$ are continuous on $\Z$. 
The bilinear form $a$ is coercive since, by Assumptions
(ii) and (iv),
$$
  a(y,y) = \langle \widetilde A[e^z](y),y\rangle
	+ \eps \langle\widetilde L[z](y),\phi\rangle
	\ge \eps\kappa_0\|y\|_\Z^2
$$
for all $y\in\Z$.
The Lax-Milgram lemma provides a unique solution $y\in\Z$ to \eqref{ex.LM}.

This defines the fixed-point operator $S:\X\times[0,1]\to\X$, $S(z,\eta)=y$.
We have to show that $S$ is continuous and compact and that $S(z,0)=0$ for
all $z\in\X$. For these properties, more specific conditions on $A$ need to be
imposed, and we refer to the following subsections. For instance, if
$\widetilde A[e^z]+\eps\widetilde L[z]$ is one-to-one, it follows that
$S(z,0)=0$, and the compactness is a consequence of the compactness of the
embedding $\Z\hookrightarrow\X$. It remains to prove
a uniform bound for all fixed points of $S(\cdot,\eta)$. A key ingredient
is the following lemma which allows us to estimate the discrete time derivative.

\begin{lemma}[Estimation of the discrete time derivative I]\label{lem.est.t}
Let $1<\alpha<2$. The following estimate holds:
$$
  \int_\Omega\bigg(\frac{\alpha_p}{\beta_p}e^{y} 
  + e^{(1-\alpha/2)y}\sum_{j=0}^{p-1}\delta_j v_{k+j}\bigg)y dx 
  \ge -C\left(1 + \sum_{j=0}^{p-1}\|v_{k+j}\|_{L^2(\Omega)}^2\right),
$$
and the constant $C>0$ only depends on $p$, $\alpha$, $\alpha_j$, and $\beta_j$
$(j=0,\ldots,p)$.
\end{lemma}

\begin{proof}
We apply the Young inequality to the second summand of the integrand:
\begin{align*}
  \int_{\Omega} & \bigg(\frac{\alpha_p}{\beta_p}ye^{y} 
  + ye^{(1-\alpha/2)y}\sum_{j=0}^{p-1}\delta_j v_{k+j}\bigg) dx \\
  &\ge \int_{\{y<0\}}\left(\frac{\alpha_p}{\beta_p}ye^y - \frac12 y^2 e^{(2-\alpha)y}
  - \frac{1}{2}\Big(\sum_{j=0}^{p-1}\delta_j v_{k+j}\Big)^2\right)dx 
  \ge -C - \frac{p}{2}\sum_{j=0}^{p-1}\delta_j^2\int_{\Omega}v_{k+j}^2 dx.
\end{align*}
The last inequality follows from the fact that the mapping
$x\mapsto (\alpha_p/\beta_p)xe^x - \frac12 x^2 e^{(2-\alpha)x}$ for $x\in\R$
is bounded from below. For this statement, we need the condition
$1<\alpha<2$.
\end{proof}

{\em Step 3: Uniform estimates.} Let $y\in\X$ be a fixed point of $S(\cdot,\eta)$
and $\eta\in[0,1]$. By construction, $y\in\Z$ solves problem \eqref{ex.y}.
With the test function $\phi=y$ in the weak formulation of \eqref{ex.y},
it follows, by Assumptions (i) and (iii) and by Lemma \ref{lem.est.t}, that
\begin{equation}\label{ex.yeps}
  \eps\kappa_1\|y\|_\Z^2 \le \langle A(e^y),y\rangle +\eps\langle L(y),y\rangle 
  \le \frac{C\eta}{\tau}\left(1 + \sum_{j=0}^{p-1}\|v_{k+j}\|_{L^2(\Omega)}^2\right).
\end{equation}
As a consequence, we obtain an $\eps$-dependent bound which is uniform in
$y\in\Z$ and $\eta\in[0,1]$ and, because of the continuous
embedding $\Z\hookrightarrow\X$,
also uniform in $y\in\X$. Thus, the fixed-point theorem of Leray-Schauder provides
the existence of a fixed point $y$ of $S(\cdot,1)$, i.e., a solution to 
\eqref{ex.y}.

{\em Step 4: Discrete entropy estimate.} We derive estimates independent of $\eps$.
For this, we use the test function $\phi=w_{k}^{\alpha-1}:=e^{(\alpha-1)y}\in\Z$
in \eqref{ex.y} (with $\eta=1$):
\begin{equation}\label{ex.wa}
  \frac{2}{\alpha\tau}\int_{\Omega} \left(\frac{\alpha_p}{\beta_p}e^{y}
  + e^{(1-\alpha/2)y}\sum_{j=0}^{p-1}\delta_j v_{k+j}\right) e^{(\alpha-1)y}dx 
  + \langle A(e^y)+\eps L(y),e^{(\alpha-1)y}\rangle = 0.
\end{equation}
The discrete time derivative is estimated as follows.

\begin{lemma}[Estimation of the discrete time derivative II]\label{lem.est.tt}
The following estimate holds:
$$
  \frac{2}{\alpha\tau}\int_{\Omega} \left(\frac{\alpha_p}{\beta_p}e^{y}
  + e^{(1-\alpha/2)y}\sum_{j=0}^{p-1}\delta_j v_{k+j}\right) e^{(\alpha-1)y}dx 
  \ge \frac{2}{\alpha\tau}\big(H[V_{k+1}]-H[V_k]\big).
$$
\end{lemma}

\begin{proof}
Using the definition $\delta_j=\alpha_j-\alpha_p\beta_j/\beta_p$ and 
definition \eqref{ex.v} for $v_{k+p}$, we can write the integrand as
\begin{align*}
  \frac{\alpha_p}{\beta_p}e^{\alpha y} 
  + e^{\alpha y/2}\sum_{j=0}^{p-1}\delta_j v_{k+j}
  &= e^{\alpha y/2}\left(\alpha_p v_{k+p} 
  + \sum_{j=0}^{p-1}\alpha_j v_{k+j}\right) \\
  &= w_k^{\alpha/2}\rho(E)v_{k} = \sigma(E)v_{k}\,\rho(E)v_{k}. \\
\end{align*}
With the G-stability of the scheme $(\rho,\sigma)$, it follows from the proof
of Proposition \ref{prop.g} that
\begin{align*}
  \frac{2}{\alpha\tau}\int_{\Omega} \left(\frac{\alpha_p}{\beta_p}e^{\alpha y}
  + e^{\alpha y/2}\sum_{j=0}^{p-1}\delta_j v_{k+j}\right)dx
  &= \frac{2}{\alpha\tau}\int_{\Omega}\sigma(E)v_{k}\,\rho(E)v_{k}dx \\
  &\ge \frac{2}{\alpha\tau}\big(H[V_{k+1}]-H[V_k]\big),
\end{align*}
ending the proof.
\end{proof}

By Assumption (iii), $\langle L(y),e^{(\alpha-1)y}\rangle\ge -C$. Therefore,
using Lemma \ref{lem.est.tt}, \eqref{ex.wa} becomes
\begin{equation}\label{ex.Heps}
  H[V_{k+1}] + \frac{\alpha\tau}{2}\langle A(e^y),e^{(\alpha-1)y}\rangle 
  \le \frac{\eps\alpha\tau C}{2} + H[V_k].
\end{equation}
This is the key inequality to derive the $\eps$-independent bounds (observe that
$\eps<1$). Depending on the properties on the operator $A$, 
the expression $\langle A(e^y),e^{(\alpha-1)y}\rangle$ may yield certain Sobolev
estimates (see Sections \ref{sec.skt} and \ref{sec.dlss} for examples). 
Assumption (i) ensures that this term is at least nonnegative, and this is 
sufficient to derive $L^\alpha$ estimates for $w_k$.

\begin{lemma}\label{lem.vw}
There exists a constant $C>0$ such that for all $k\in\N$,
$$
  \|v_{k+p}\|_{L^2(\Omega)} + \|w_k^{\alpha/2}\|_{L^2(\Omega)} \le C,
$$
where $C>0$ does not depend on $\eps$.
\end{lemma}

\begin{proof}
The positive definiteness of the matrix $G$ (with some constant $C_G>0$)
implies that
\begin{align*}
  H[V_{k+1}] &= \frac12\sum_{j=0}^{p-1}G_{ij}\int_{\Omega}v_{k+1+i}v_{k+1+j}dx
  \ge \frac{C_G}{2}\sum_{j=0}^{p-1}\int_{\Omega}v_{k+1+j}^2 dx \\
  &= \frac{C_G}{2}\int_{\Omega}v_{k+p}^2 dx 
	+ \frac{C_G}{2}\sum_{j=1}^{p-1}\int_{\Omega}v_{k+j}^2 dx.
\end{align*}
In view of \eqref{ex.Heps}, this provides a uniform estimate for $v_{k+p}$ in
$L^2(\T^d)$. Furthermore,
$$
  \int_{\Omega}w_{k}^\alpha dx
  = \int_{\Omega}(\sigma(E)v_{k})^2 dx
  \le (p+1)\int_{\Omega}\sum_{j=0}^p\beta_j^2 v_{k+j}^2 dx,
$$
and we conclude a uniform estimate for $w_{k}$ in $L^\alpha(\Omega)$. 
\end{proof}

{\em Step 5: Limit $\eps\to 0$ in \eqref{ex.y}.}
Set $w_\eps=w_{k}=e^y$, $v_\eps=v_{k+p}$, and $y_\eps=y$.
Then, using Lemma \ref{lem.vw} and \eqref{ex.yeps}, 
\begin{equation}\label{ex.est}
  \|v_\eps\|_{L^2(\Omega)} + \|w_\eps^{\alpha/2}\|_{L^2(\Omega)}
  + \sqrt{\eps}\|y_\eps\|_\Z \le C.
\end{equation}
The limit $\eps\to 0$ depends on the specific structure of $A$
and cannot be detailed here without further assumptions. We refer to
the examples presented below. If we are able to prove that 
$v_\eps\to v$, $w_\eps\to w$ in appropriate spaces
for some $v$, $w$ and if the limit functions satisfy \eqref{g.disc},
we can set $v_{k+p}:=v$ and $w_k:=w$.

{\em Step 6: Discrete entropy dissipation.} The weak convergence 
$v_\eps\rightharpoonup v_{k+p}$ in $L^2(\Omega)$ for a subsequence
(see \eqref{ex.est})
and the lower semi-continuity of $u\mapsto \|u\|_{L^2(\Omega)}^2$ on $L^2(\Omega)$ yield
\begin{align*}
  \phantom{x}&\liminf_{\eps\to 0} H[v_{k+1},\ldots,v_{k+p-1},v_\eps] 
  = \sum_{i,j=0}^{p-2}G_{ij}\int_{\Omega}v_{k+1+i}v_{k+1+j}dx \\
  &\phantom{xx}{}
	+ 2\lim_{\eps\to 0}\sum_{j=0}^{p-2}G_{p-1,j}\int_{\Omega}v_\eps v_{k+1+j}dx
  + \liminf_{\eps\to 0}G_{p-1,p-1}\int_{\Omega}v_\eps^2 dx \\
  &\ge \sum_{i,j=0}^{p-2}G_{ij}\int_{\Omega}v_{k+1+i}v_{k+1+j}dx
  + 2\sum_{j=0}^{p-2}G_{p-1,j}\int_{\Omega}v_{k+p} v_{k+1+j}dx 
  + G_{p-1,p-1}\int_{\Omega}v_{k+p}^2 dx \\
  &= H[v_{k+1},\ldots,v_{k+p-1},v_{k+p}] = H[V_{k+1}].
\end{align*}
Assuming that 
\begin{equation}\label{ex.liminf}
  \liminf_{\eps\to 0}\langle A(e^{y_\eps}),e^{(\alpha-1)y_\eps}\rangle 
  \ge \langle A(e^y),e^{(\alpha-1)y}\rangle, 
\end{equation}
the limit $\eps\to 0$ in \eqref{ex.Heps} shows that
$$
  H[V_{k+1}] + \frac{\alpha\tau}{2}\langle A(e^y),e^{(\alpha-1)y}\rangle \le H[V_k],
$$
which is the desired discrete entropy dissipation inequality.


\subsection{A cross-diffusion population system}\label{sec.skt}

We make the arguments of the previous subsection rigorous for the
cross-diffusion system \eqref{1.skt1}-\eqref{1.skt3}.
Let $d\le 3$ and define the spaces $\HH=L^2(\T^d)^2$, $\V=H^1(\T^d)^2$, 
$\Z=H^2(\T^d)^2$, $\X=W^{1,4}(\T^d)^2$, and
$D(A)=\{u=(u^{(1)},u^{(2)})\in W^{1,4}(\T^d)^2:u^{(1)}>0$, $u^{(2)}>0\}$. 
The operator $A:D(A)\to\V'$ is given by
\begin{align*}
  \langle A(u),\phi\rangle 
  &= \int_{\T^d}\big((d_1+a_1 u^{(1)}+u^{(2)})\na u^{(1)}\cdot\na\phi^{(1)}
  + u^{(1)}\na u^{(2)}\cdot\na\phi^{(1)} \\
  &\phantom{xx}{}+ (d_2+a_2 u^{(2)}+u^{(1)})\na u^{(2)}\cdot\na\phi^{(2)}
  + u^{(2)}\na u^{(1)}\cdot\na\phi^{(2)}\big)dx
\end{align*}
for $u=(u^{(1)},u^{(2)})\in D(A)$ and $\phi=(\phi^{(1)},\phi^{(2)})\in\V$.
We have to verify Assumptions (i)-(iv) stated in Section \ref{sec.gen}. 
We show first that (i) is satisfied.
In the following, we define $f(u)=(f(u^{(1)}),f(u^{(2)}))$ for arbitrary functions 
$f:\R^2\to\R$ and $u=(u^{(1)},u^{(2)})$.

\begin{lemma}\label{lem.assA}
Let $1<\alpha\le 2$, $u\in D(A)$, and $4a_1 a_2\ge\max\{a_1,a_2\}+1$. Then 
\begin{align*}
  \langle A(u),\log u\rangle 
  &\ge 4\int_{\T^d}\big(d_1|\na\sqrt{u^{(1)}}|^2 + d_2|\na\sqrt{u^{(2)}}|^2\big)dx, \\ 
  \langle A(u),u^{\alpha-1}\rangle &\ge \frac{4(\alpha-1)}{\alpha^2}\int_{\T^d}\big(
  d_1|\na (u^{(1)})^{\alpha/2}|^2 + d_2|\na (u^{(2)})^{\alpha/2}|^2\big)dx.
\end{align*}
\end{lemma}

\begin{proof}
The first inequality follows from
\begin{align*}
  \langle A(u),\log u\rangle
  &= \int_{\T^d}\bigg(d_1\frac{|\na u^{(1)}|^2}{u^{(1)}} 
  + d_2\frac{|\na u^{(2)}|^2}{u^{(2)}}
  + a_1|\na u^{(1)}|^2 + a_2|\na u^{(2)}|^2 \\
  &\phantom{xx}{}+ 4|\na\sqrt{u^{(1)}u^{(2)}}|^2\bigg)dx.
\end{align*}
To prove the second inequality, we calculate
\begin{align*}
  \langle A(u),u^{\alpha-1}\rangle
  &= (\alpha-1)\int_{\T^d}\big((d_1+a_1u^{(1)}+u^{(2)})
  (u^{(1)})^{\alpha-2}|\na u^{(1)}|^2
  + (u^{(1)})^{\alpha-1}\na u^{(2)}\cdot\na u^{(1)} \\
  &\phantom{xx}{}+ (d_2+a_2u^{(2)}+u^{(1)})(u^{(2)})^{\alpha-2}|\na u^{(2)}|^2
  + (u^{(2)})^{\alpha-1}\na u^{(1)}\cdot\na u^{(2)}\big)dx.
\end{align*}
For $u^{(1)}\le u^{(2)}$, we find that
\begin{align*}
  \big((u^{(1)})^{\alpha-1} & + (u^{(2)})^{\alpha-1}\big)\na u^{(1)}\cdot\na u^{(2)} 
  \ge -a_1 (u^{(1)})^{\alpha-1}|\na u^{(1)}|^2 \\
  &\phantom{xx}{}- \frac{1}{4a_1}(u^{(1)})^{\alpha-1}|\na u^{(2)}|^2 
  - \frac14 (u^{(2)})^{\alpha-1}|\na u^{(2)}|^2 
  - u^{(2)} (u^{(2)})^{\alpha-2}|\na u^{(1)}|^2 \\
  &\ge -a_1 (u^{(1)})^{\alpha-1}|\na u^{(1)}|^2 - \left(\frac14+\frac{1}{4a_1}\right)
  (u^{(2)})^{\alpha-1}|\na u^{(2)}|^2 - u^{(2)}(u^{(1)})^{\alpha-2}|\na u^{(1)}|^2.
\end{align*}
Here, we have used the inequalities $(u^{(1)})^{\alpha-1}\le (u^{(2)})^{\alpha-1}$
and $(u^{(2)})^{\alpha-2}\le (u^{(1)})^{\alpha-2}$, noting that $1<\alpha\le 2$.
In a similar way, it follows for $u^{(1)}>u^{(2)}$ that
\begin{align*}
  \big((u^{(1)})^{\alpha-1}&+(u^{(2)})^{\alpha-1}\big)\na u^{(1)}\cdot\na u^{(2)}
  \ge -a_2 (u^{(2)})^{\alpha-1}|\na u^{(2)}|^2 \\
  &\phantom{xx}{}- \left(\frac14+\frac{1}{4a_2}\right)
  (u^{(1)})^{\alpha-1}|\na u^{(1)}|^2 - u^{(1)}(u^{(2)})^{\alpha-2}|\na u^{(2)}|^2.
\end{align*}
This shows that
\begin{align*}
  \langle A(u),u^{\alpha-1}\rangle
  &\ge \frac{4(\alpha-1)}{\alpha^2}
  \int_{\T^d}\big(d_1|\na (u^{(1)})^{\alpha/2}|^2 
  + d_2|\na (u^{(2)})^{\alpha/2}|^2\big)dx \\
  &\phantom{xx}{}
  + \frac{\alpha-1}{4a_2}(4a_1a_2-a_2-1)\int_{\T^2}
  (u^{(1)})^{\alpha-1}|\na u^{(1)}|^2 dx \\
  &\phantom{xx}{}
  + \frac{\alpha-1}{4a_1}(4a_1a_2-a_1-1)\int_{\T^2}
  (u^{(2)})^{\alpha-1}|\na u^{(2)}|^2 dx
  \ge 0,
\end{align*}
since the assumption on $a_1$, $a_2$ implies that
$4a_1a_2-a_2-1\ge 0$ and $4a_1a_2-a_1-1\ge 0$.
\end{proof}

The regularization $L:\Z\to\Z'$ is defined by
\begin{equation}\label{skt.L}
  \langle L(y),\phi\rangle = \sum_{m=1}^2\int_{\T^d}\big(\Delta y^{(m)}\Delta\phi^{(m)}
  + |\na y^{(m)}|^2\na y^{(m)}\cdot\na\phi^{(m)} + y^{(m)} \phi^{(m)}\big)dx
\end{equation}
for $y=(y^{(1)},y^{(2)})$, $\phi=(\phi^{(1)},\phi^{(2)})\in\Z$. 
We show that it fulfills Assumption (iii).

\begin{lemma}\label{lem.assL}
It holds for all $y\in\Z$
$$
  \langle L(y),e^{(\alpha-1)y}\rangle \ge -\frac{1}{e(\alpha-1)}, \quad
  \langle L(y),y\rangle \ge \kappa_1\|y\|_\Z^2,
$$
where $\kappa_1>0$ only depends on the Poincar\'e constant for periodic functions 
with vanishing integral mean.
\end{lemma}

\begin{proof}
A a straightforward computation shows that for $m=1,2$,
\begin{align*}
  \Delta y^{(m)} &\Delta (e^{(\alpha-1)y^{(m)}}) 
  + |\na y^{(m)}|^2\na y^{(m)}\cdot\na (e^{(\alpha-1)y^{(m)}}) \\
  &= 4(\alpha-1)e^{(\alpha-1)y^{(m)}}\left(\frac{\Delta e^{y^{(m)}/2}}{e^{y^{(m)}/2}}
  - (2-\alpha)\left|\frac{\na e^{y^{(m)}/2}}{e^{y^{(m)}/2}}\right|^2\right)^2 \\
  &\phantom{xx}{}
  + 4(\alpha^2-1)(3-\alpha)\left|\frac{\na e^{y^{(m)}/2}}{e^{y^{(m)}/2}}\right|^4 
  \ge 0.
\end{align*}
Then the first inequality follows from the fact that the mapping
$x\mapsto xe^{(\alpha-1)x}$, $x\in\R$, is bounded from below by $-1/(e(\alpha-1))$.
The second inequality is a consequence of
\begin{align}
  \|y\|_{H^2(\T^2)^2}^2 
  &= \sum_{m=1}^2\int_{\T^d}\big(\|\na^2 y^{(m)}\|^2+|\na y^{(m)}|^2
  +(y^{(m)})^2\big)dx \nonumber \\
  &\le \sum_{m=1}^2\int_{\T^d}\big((C^2+1)\|\na^2 y^{(m)}\|^2+(y^{(m)})^2\big)dx 
  \nonumber \\
  &\le (C^2+1)\sum_{m=1}^2\int_{\T^d}\big((\Delta y^{(m)})^2+(y^{(m)})^2\big)dx
  \le (C^2+1)\langle L(y),y\rangle, \label{H2norm}
\end{align}
where $C>0$ is the Poincar\'e constant. 
\end{proof}

The ``linearization'' of $A$ is defined by ``freezing'' the diffusion coefficients:
\begin{align*}
  \langle\widetilde A[e^z](y),\phi\rangle
  &= \int_{\T^d}\big((d_1+a_1 e^{z^{(1)}}+e^{z^{(2)}})e^{z^{(1)}}\na y^{(1)} 
  + e^{z^{(1)}+z^{(2)}}\na y^{(2)}\big)\cdot\na\phi^{(1)} dx \\
  &\phantom{xx}{}
  +\int_{\T^d}\big((d_2+a_2 e^{z^{(2)}}+e^{z^{(1)}})e^{z^{(2)}}\na y^{(2)} 
  + e^{z^{(1)}+z^{(2)}}\na y^{(1)}\big)\cdot\na\phi^{(2)} dx,
\end{align*}
for $z=(z^{(1)},z^{(2)})\in\X=W^{1,4}(\T^d)^2$, $y$, $\phi\in\V$.
This operator satisfies Assumption (ii):
\begin{align*}
  \langle \widetilde A[e^z](y),y\rangle
	&= \int_{\T^d}\big((d_1+a_1 e^{2z^{(1)}})|\na y^{(1)}|^2
	+ (d_2+a_2 e^{2z^{(2)}})|\na y^{(12)}|^2 \\
	&\phantom{xx}{}+ e^{z^{(1)}+z^{(2)}}|\na(y^{(1)}+y^{(2)})|^2\big)dx \ge 0,
  \quad y\in\Z,
\end{align*}
Similarly, we define $\widetilde L[z]$ for $z\in\X$ and $y$, $\phi\in\Z$ by
$$
  \langle\widetilde L[z](y),\phi\rangle 
  = \sum_{m=1}^2\int_{\T^d}\big(\Delta y^{(m)}\Delta\phi^{(m)}
  + |\na z^{(m)}|^2\na y^{(m)}\cdot\na\phi^{(m)} + y^{(m)} \phi^{(m)}\big)dx,
$$
fulfilling Assumption (iv).

In Step 2 in Section \ref{sec.gen}, we have defined a fixed-point operator
$S:\X\times[0,1]\to\X$. It is not difficult to show that this operator is
continuous and compact, taking into account the compactness of the embedding
$\Z=H^2(\T^d)^2\hookrightarrow\X=W^{1,4}(\T^d)^2$. In order to show that
$S(z,0)=0$ for $z\in\Z$, we write
$$
  \widetilde A[e^z](y) = -\diver(D(z)\na y),
$$
where the diffusion matrix 
$$
  D(z) = \begin{pmatrix} 
  (d_1 + a_1 e^{z^{(1)}} + e^{z^{(2)}})e^{z^{(1)}} & e^{z^{(1)}+z^{(2)}} \\
  e^{z^{(1)}+z^{(2)}} & (d_2 + a_2 e^{z^{(2)}} + e^{z^{(1)}})e^{z^{(2)}}
  \end{pmatrix},
$$
is symmetric and positive definite. Therefore,
$\widetilde A[e^z](y)+\eps\widetilde L[z](y)$ is one-to-one, showing that
$S(z,0)=0$. We infer from Steps 1-3 in Section \ref{sec.gen}
that there exists a weak solution $y\in\Z$ to \eqref{ex.y}. 
It remains to derive discrete entropy estimates (independent of
$\eps$) and to perform the limit $\eps\to 0$.
Employing the test function 
$\phi=w_{k}^{\alpha-1}=e^{(\alpha-1)y}\in\Z$
in the weak formulation of \eqref{ex.y} and taking into account Lemmas
\ref{lem.est.tt}, \ref{lem.assA}, and \ref{lem.assL}, we obtain the discrete
entropy estimate
\begin{equation}\label{skt.ent}
  H[V_{k+1}] + \frac{2\tau}{\alpha}(\alpha-1)\int_{\T^d}\big(
  d_1|\na (w_{k}^{(1)})^{\alpha/2}|^2 + d_2|\na (w_{k}^{(1)})^{\alpha/2}|^2\big)dx
  \le \frac{\eps\alpha\tau}{2e(\alpha-1)} + H[V_k].
\end{equation}

This estimate and \eqref{H2norm} are sufficient to perform the limit
$\eps\to 0$. Set $w_\eps=w_k=e^y$, $v_\eps=v_{k+p}$ (defined in \eqref{ex.v}),
and $y_\eps=y$. Because of Lemma \ref{lem.vw} and \eqref{skt.ent}, 
we have the $\eps$-independent bounds
$$
  \|v_{\eps}^{(j)}\|_{L^2(\T^d)} + \|(w_{\eps}^{(j)})^{\alpha/2}\|_{H^1(\T^d)} 
  + \sqrt{\eps}\|y_{\eps}^{(j)}\|_{H^2(\T^d)} \le C, \quad j=1,2.
$$
By compactness, there exist subsequences, which are not relabeled, such that,
as $\eps\to 0$, for $j=1,2$,
\begin{align*}
  v_{\eps}^{(j)} \rightharpoonup v^{(j)} &\quad\mbox{weakly in }L^2(\T^d), \\
  (w_{\eps}^{(j)})^{\alpha/2}\rightharpoonup z^{(j)} 
  &\quad\mbox{weakly in }H^1(\T^d), \\
  (w_{\eps}^{(j)})^{\alpha/2}\to z^{(j)} &\quad\mbox{strongly in }L^6(\T^d), \\
  \eps y_{\eps}^{(j)} \to 0 &\quad\mbox{strongly in }H^2(\T^d)
  \mbox{ and in }W^{1,4}(\T^d).
\end{align*}
The last limit implies that $\eps L(y_\eps)\to 0$ in $H^{-2}(\T^d)^2$.

In view of $0<\exp(\alpha y_{\eps}^{(j)}/2)=(w_{\eps}^{(j)})^{\alpha/2}
=\sigma(E)v_{\eps}^{(j)}$
and the linearity of the operator $\sigma(E)$, it follows that
$0\le z^{(j)}=\sigma(E)v^{(j)}$, where $\sigma(E)v^{(j)}
=\beta_p v^{(j)}+\sum_{m=0}^{p-1}\beta_m v_{k+m}^{(j)}$.
This allows us to define $w:=z^{2/\alpha}$, where $z=(z^{(1)},z^{(2)})$.
Since $1<\alpha<2$, $w_\eps^{1-\alpha/2}\to w^{1-\alpha/2}$ strongly in $L^6(\T^d)^2$.
We infer that
\begin{align*}
  \frac{2}{\alpha\tau}\left(\frac{\alpha_p}{\beta_p}e^{y_\eps} + e^{(1-\alpha/2)y_\eps}
  \sum_{m=0}^{p-1} \delta_m v_{k+m} \right)
  &= \frac{2}{\alpha\tau}e^{(1-\alpha/2)y_\eps}
  \left(\frac{\alpha_p}{\beta_p}e^{\alpha y_\eps/2} 
  + \sum_{m=0}^{p-1} \delta_m v_{k+m}\right) \\
  &= \frac{2}{\alpha\tau} w_\eps^{1-\alpha/2}\rho(E)v_\eps \\
  &\rightharpoonup \frac{2}{\alpha\tau} w^{1-\alpha/2}\rho(E)v 
  \quad\mbox{weakly in }L^{3/2}(\T^d),
\end{align*}
where $\rho(E)v=\alpha_p v+\sum_{m=0}^{p-1}\alpha_m v_{k+m}$.

It remains to perform the limit $\eps\to 0$ in the term involving $A$.
We find that, for $j=1,2$,
\begin{equation}\label{skt.naw}
  \na w_{\eps}^{(j)} = \frac{2}{\alpha}(w_{\eps}^{(j)})^{1-\alpha/2}
  \na (w_{\eps}^{(j)})^{\alpha/2}
  \rightharpoonup \frac{2}{\alpha}(w^{(j)})^{1-\alpha/2}\na (w^{(j)})^{\alpha/2}
  \quad\mbox{weakly in }L^{3/2}(\T^d).
\end{equation}
Since $\na w^{(j)}_\eps\to \na w^{(j)}$ in the sense of distributions,
\eqref{skt.naw} shows that $\na w^{(j)}_\eps\rightharpoonup \na w^{(j)}$
weakly in $L^{3/2}(\T^d)$. We conclude that 
\begin{align*}
  \langle A(w_\eps),\phi\rangle
  &\to \frac{2}{\alpha}
  \int_{\T^d}\big((d_1 + a_1w^{(1)} + w^{(2)})
  \na w^{(1)}\cdot\na\phi^{(1)} + w^{(1)}\na w^{(2)}\cdot\na\phi^{(1)} \\
  &\phantom{xx}{}+ (d_2 + a_2w^{(2)} + w^{(1)})
  \na w^{(2)}\cdot\na\phi^{(2)} + w^{(2)}\na w^{(1)}\cdot\na\phi^{(2)}\big)dx
\end{align*}
for all $\phi\in W^{1,\infty}(\T^d)^2$.
The limit $\eps\to 0$ in \eqref{ex.y} then
yields \eqref{1.d1}-\eqref{1.d2}. 
Finally, applying the limes inferior to \eqref{skt.ent}
and using the weak convergence \eqref{skt.naw} and the lower semi-continuity of
$u\mapsto\|\na u\|_{L^2(\T^d)}^2$ on $H^1(\T^d)$, inequality \eqref{1.skt.ent}
follows. Setting $v_{k+p}:=v$ and $w_{k+p}:=w$,
this finishes the proof of Theorem \ref{thm.skt}.


\subsection{The DLSS equation}\label{sec.dlss}

We apply the general scheme \eqref{g.disc} to the DLSS equation \eqref{1.dlss}.
For this, let $d\le 3$ and define $\HH=L^2(\T^d)$, $\V=\Z=H^2(\T^d)$,
$\X=W^{1,4}(\T^d)$, $D(A)=\{u\in H^2(\T^d):u>0$ in $\T^d\}$, and $A:D(A)\to\V'$ by
$$
  \langle A(u),\phi\rangle = \langle\na^2:(u\na^2\log u),\phi \rangle
  = \int_{\T^d}\big(\na^2 u
  - 4\na\sqrt{u}\otimes\na\sqrt{u}\big):\na^2\phi dx, 
$$
for $u\in D(A)$, $\phi\in\V$, where 
the tensor product $\na\sqrt{u}\otimes\na\sqrt{u}$ consists of the components
$(\pa \sqrt{u}/\pa x_i)(\pa \sqrt{u}/\pa x_j)$. 
Observe that $u\in\V$ implies that $\sqrt{u}\in W^{1,4}(\T^d)$ by the
Lions-Villani lemma (see the version in \cite[Lemma 26]{BJM11}) such that
the integral on the right-hand side of the definition of
$A$ is well defined. We need to verify Assumptions (i)-(iv) of Section \ref{sec.gen}.

By Lemma 2.2 of \cite{JuMa08}, it holds that
\begin{align}
  \langle A(u),\log u\rangle &\ge \kappa_0\int_{\T^d}(\Delta u^{1/2})^2 dx, 
  \nonumber \\
  \langle A(u),u^{\alpha-1}\rangle &\ge \kappa_\alpha\int_{\T^d}
  (\Delta u^{\alpha/2})^2 dx \label{dlss.A}
\end{align}
for all $1<\alpha<2(d+1)/(d+2)$ and $u\in D(A)$, where $\kappa_\alpha>0$ for
$\alpha\ge 1$ depends only on $\alpha$ and the space dimension $d$. 
This shows Assumption (i). The ``linearization'' of $A$ is defined by
$\widetilde A[v](u)=\na^2:(v\na^2\log u)$ for $u$, $v\in D(A)$.
The regularization $L$ and its ``linearization'' are defined similarly as 
in Section \ref{sec.skt}:
$$
  L(y) = \Delta^2 y - \diver(|\na y|^2\na y) + y, \quad
  \widetilde L[z](y) = \Delta^2 y - \diver(|\na z|^2\na y) + y
$$
for $y$, $z\in\V$ with $u=e^y$, $v=e^z$. 
By Lemma \ref{lem.assL}, Assumption (iii) is satisfied.
Moreover, Assumptions (ii) and (iv) hold as well.

From Steps 1-3 of Section \ref{sec.gen}, we infer that there exists a weak
solution $y\in\V$ to \eqref{ex.y}. The discrete entropy estimate
follows from Lemmas \ref{lem.est.tt} and \ref{lem.assL} and estimate \eqref{dlss.A}:
$$
  H[V_{k+1}] + \frac{\alpha\tau}{2}\kappa_\alpha\int_{\T^d}
  (\Delta w_k^{\alpha/2})^2 dx \le \frac{\eps\tau}{2e(\alpha-1)} + H[V_k].
$$
Together with the $L^2$-bound for $w_k^{\alpha/2}$ from Lemma \ref{lem.vw},
we obtain the following $\eps$-inde\-pen\-dent estimates
$$
  \|v_\eps\|_{L^2(\T^d)} + \|w_\eps^{\alpha/2}\|_{H^2(\T^d)} 
  + \sqrt{\eps}\|y_\eps\|_{H^2(\T^d)} \le C,
$$
where $v_\eps:=v_{k+p}$ (defined in \eqref{ex.v}), $w_\eps:=w_k$, and $y_\eps:=y$.
Arguing as in Section \ref{sec.skt}, there exist subsequences such that,
as $\eps\to 0$,
\begin{align*}
  v_\eps\rightharpoonup v &\quad\mbox{weakly in }L^2(\T^d), \\
  w_\eps^{\alpha/2}\rightharpoonup w^{\alpha/2} &\quad\mbox{weakly in }H^2(\T^d), \\
  w_\eps^{\alpha/2}\to w^{\alpha/2} &\quad\mbox{strongly in }L^{\infty}(\T^d), \\
  \eps y_\eps \to 0 &\quad\mbox{strongly in }H^2(\T^d),
\end{align*}
where $w^{\alpha/2}=\sigma(E)v$, and $\eps L(y_\eps)\to 0$ in $H^{-2}(\T^d)$.

We perform the limit $\eps\to 0$ in the fourth-order operator.
By the Lions-Villani lemma \cite[Lemma 26]{BJM11},
$$
  \|w_\eps^{\alpha/4}\|_{W^{1,4}(\T^d)}
  \le C\|w_\eps^{\alpha/2}\|_{H^2(\T^d)} \le C.
$$
Hence, because of
$$
  (w_\eps^{\alpha/4}) \mbox{ is bounded in }W^{1,4}(\T^d), \quad
  w_\eps^{\alpha/2}\to w^{\alpha/2} \mbox{ strongly in }H^1(\T^d),
$$
and $\alpha/4<1/2<\alpha/2$, Proposition A.1 of
\cite{JuMi09} shows that
$$
  |\na w_\eps^{1/2}| \to |\na w^{1/2}| \quad\mbox{strongly in }L^{2\alpha}(\T^d).
$$
Then the limit $\eps\to 0$ gives
\begin{align*}
  \na^2:(e^{y_\eps}\na^2 y_\eps)
  &= \na^2:\left(\frac{2}{\alpha}w_\eps^{1-\alpha/2}\na^2 w_\eps^{\alpha/2} 
  - 2\alpha\na w_\eps^{1/2}\otimes\na w_\eps^{1/2}\right) \\
  &\rightharpoonup \na^2:\left(\frac{2}{\alpha}w^{1-\alpha/2}\na^2 w^{\alpha/2} 
  - 2\alpha\na w^{1/2}\otimes\na w^{1/2}\right)
\end{align*}
weakly in $W^{-2,\alpha}(\T^d)$.
Therefore, passing to the limit $\eps\to 0$ in \eqref{ex.y} yields \eqref{ex.weak0}.
Finally, the discrete entropy dissipation inequality \eqref{ex.H} follows as
at the end of the previous subsection from the weak convergence of $w_\eps^{\alpha/2}$
to $w^{\alpha/2}$ in $H^2(\T^d)$ and the lower semi-continuity of
$u\mapsto\|\Delta u\|_{L^2(\T^d)}^2$ on $H^2(\T^d)$.


\section{Convergence rate}\label{sec.second}

In this section, we prove Theorem \ref{thm.second}. Let $0<\tau<1$. 
The idea of the proof is to estimate the difference $\widehat v(t)=v(t)-\delta_I(t)$ 
(see \eqref{deltaD} and \cite[Theorem V.6.10]{HaWa91}). 
Setting $B(v)=\frac{\alpha}{2} v^{1-2/\alpha}A(v^{2/\alpha})$,
we rewrite equation \eqref{1.eqv} for the exact solution as
\begin{equation}\label{3.Bv}
  \tau v_t(t_{k+2}) + \tau B(v(t_{k+2})) = 0.
\end{equation}
With the definitions \eqref{deltaD} of $\delta_D$ and $\delta_I$, 
this can be formulated as
\begin{equation}\label{3.vd}
  \rho(E)\widehat v(t_k) 
	+ \tau B\big(\sigma(E)\widehat v(t_k)-\widehat\eps(t_k)\big) = -\widehat\delta(t_k),
\end{equation}
where $\widehat\delta(t)=\delta_D(t)-\rho(E)\delta_I(t)$ and
$\widehat\eps(t)=\delta_I(t)-\sigma(E)\delta_I(t)$.

We derive bounds for $\widehat\delta$ and $\widehat\eps$.
A Taylor expansion of $\delta_I$ at $t=t_{k+1}$ and $t=t_{k+2}$ 
around $t_{k}$ and the condition $\sigma(1)=1$ yields
\begin{align*}
  \sigma(E)\delta_I(t_k) &= \beta_0\delta_I(t_k)
	+ \beta_1\big(\delta_I(t_k) + \tau\delta'_I(t_k) 
	+ \tfrac12\tau^2\delta''_I(t_k)\big) \\
	&\phantom{xx}{}
	+ \beta_2\big(\delta_I(t_k) + 2\tau\delta'_I(t_k) + 2\tau^2\delta''_I(t_k)\big)
	+ O(\tau^3) \\
	& = \delta_I(t_k) + (\beta_1+2\beta_2)\tau\delta'_I(t_k) 
	+ \frac12(\beta_1+4\beta_2)\tau^2\delta''_I(t_k) + O(\tau^3).
\end{align*}
Since $\delta_I(t_k) = O(\tau^2)$, it follows that
\begin{equation}\label{3.eps}
  \widehat\eps(t_k) = \delta_I(t_{k}) - \sigma(E)\delta_I(t_k) = O(\tau^3). 
\end{equation}
In a similar way, a Taylor expansion of $\delta_I$ at $t=t_{k+1}$ and $t=t_{k+2}$ 
around $t_k$ gives, because of $\rho(1)=0$,
\begin{align*}
  \rho(E)\delta_I(t_k) &= (\alpha_0+\alpha_1+\alpha_2)\delta_I(t_k)
	+ (\alpha_1+2\alpha_2)\tau\delta'_I(t_k) + \frac12(\alpha_1+4\alpha_2)\tau^2
	\delta''_I(t_k) + O(\tau^3) \\
	&= -\tau(\alpha_1+2\alpha_2)\delta'_I(t_k) 
	- \frac12(\alpha_1+4\alpha_2)\tau^2\delta''_I(t_k) + O(\tau^3)
	= O(\tau^3).
\end{align*}
Because of $\delta_D(t)=O(\tau^3)$ (see Section \ref{sec.g.ene}) we infer that
\begin{equation}\label{3.delta}
  \widehat\delta(t_k) = \delta_D(t_k) - \rho(E)\delta_I(t_k) = O(\tau^3).
\end{equation}

Concerning the first time step of the scheme, we observe that, by 
a Taylor expansion,
$$
  \tau v_t(t_1) = v(t_1)-v(t_0)+f_0, \quad 
	\|f_0\|\le \frac{\tau^2}{2}\|v_{tt}\|_{L^\infty(0,T;\HH)},
$$
and \eqref{3.Bv} becomes
\begin{equation}\label{3.1}
  v(t_1)-v(t_0) + \tau B(v(t_1)) = -f_0.
\end{equation}
The difference of scheme \eqref{euler}, formulated as
$(v_1-v_0) + \tau B(v_1) = 0$, and \eqref{3.1} becomes
$$
  e_1-e_0 + \tau\big(B(v_1)-B(v(t_1)\big) = f_0,
$$
where $e_j=v_j-v(t_j)$, $j=0,1$.
Then, taking this equation in the scalar product with $e_1$
and employing $e_0=0$ and the one-sided Lipschitz condition of $B$, we obtain
$$
  \|e_1\|^2 \le \kappa_1\tau\|e_1\|^2 + (f_0,e_1)
	\le \kappa_1\tau\|e_1\|^2 + \frac12\|f_0\|^2 + \frac12\|e_1\|^2.
$$
The above error estimate for $f_0$ yields $\|e_1\|\le C\tau^2$ 
if $\tau < 1/(2\kappa_1)$, where $C>0$ depends on $\|v_{tt}\|_{L^\infty(0,T;\HH)}$.

The difference of the equations
$\rho(E)v_k + \tau B(\sigma(E)v_k) = 0$ and \eqref{3.vd} leads to
the error equations for $e_k=v_k-\widehat v(t_k)$ ($k\ge 2$):
$$
  \rho(E)e_k + \tau\big(B(\sigma(E)v_k)-B(\sigma(E)\widehat v(t_k)
	-\widehat\eps(t_k))\big) = -\widehat\delta(t_k), \quad k\ge 0.
$$
Taking these equations in the dual product with $\sigma(E)e_k+\widehat\eps(t_k)$
gives
\begin{align}
  \big(\rho(E)e_k,\sigma(E) e_k\big) 
	&= -\tau\big\langle
	B(\sigma(E)v_k)-B(\sigma(E)\widehat v(t_k)-\widehat\eps(t_k)),
	\sigma(E)e_k+\widehat\eps(t_k)\big\rangle \nonumber \\
	&\phantom{xx}{}- \big(\rho(E)e_k,\widehat\eps(t_k)\big) 
	- \big(\widehat\delta(t_k),\sigma(E)e_k+\widehat\eps(t_k)\big). \label{3.ee}
\end{align}
We estimate these expressions term by term. By the G-stability, the left-hand
side becomes
$$
  \big(\rho(E)e_k,\sigma(E) e_k\big) 
	\ge \frac12\|E_{k+1}\|_G^2 - \frac12\|E_k\|_G^2,
$$
where $E_k=(e_k,e_{k+1})$. With the one-sided Lipschitz condition for $B$
and the Cauchy-Schwarz inequality,
it follows for the first term of the right-hand side of \eqref{3.ee} that
\begin{align*}
  -\tau\big\langle &	B(\sigma(E)v_k)-B(\sigma(E)\widehat v(t_k)+\widehat\eps(t_k)),
	\sigma(E)e_k+\widehat\eps(t_k)\big\rangle \\
	&\le \kappa_1\tau\|\sigma(E)e_k+\widehat\eps(t_k)\|^2
	\le 2\kappa_1\tau\|\sigma(E)e_k\|^2 + 2\kappa_1\tau\|\widehat\eps(t_k)\|^2.
\end{align*}
Because of the positive definiteness of the matrix $G$, we conclude that
$\|E_k\|_G^2\ge C(\|e_k\|^2+\|e_{k+1}\|^2)$ which implies that,
for some $C>0$ which depends on $\beta_j$ and $G$,
\begin{equation}\label{3.equi}
  \|\sigma(E)e_k\|^2 
	\le C(\|e_k\|^2 + \|e_{k+1}\|^2 + \|e_{k+2}\|^2) 
	\le C(\|E_{k}\|_G^2 + \|E_{k+1}\|_G^2).
\end{equation}
Then, using \eqref{3.eps}, we obtain
$$
  -\tau\big\langle B(\sigma(E)v_k)-B(\sigma(E)\widehat v(t_k)+\widehat\eps(t_k)),
	\sigma(E)e_k+\widehat\eps(t_k)\big\rangle
	\le C\tau (\|E_{k}\|_G^2 + \|E_{k+1}\|_G^2) + C\tau^7,
$$
where here and in the following, $C>0$ denotes a generic constant independent
of $k$ and $\tau$.
In a similar way, the second term on the right-hand side of \eqref{3.ee} 
is estimated by
$$
  -\big(\rho(E)e_k,\widehat\eps(t_k)\big)
	\le \frac{\tau}{2}\|\rho(E)e_k\|^2 + \frac{1}{2\tau}\|\widehat\eps(t_k)\|^2
	\le C\tau(\|E_{k}\|_G^2 + \|E_{k+1}\|_G^2) + C\tau^5.
$$
Finally, we use \eqref{3.eps} and \eqref{3.delta} to estimate 
the last term in \eqref{3.ee}:
\begin{align*}
  -\big(\widehat\delta(t_k),\sigma(E)e_k+\widehat\eps(t_k)\big)
	&\le \frac{1}{2\tau}\|\widehat\delta(t_k)\|^2
	+ \frac{\tau}{2}\|\sigma(E)e_k\|^2 + \|\widehat\delta(t_k)\|\,\|\widehat\eps(t_k)\| \\
	&\le C\tau(\|E_{k}\|_G^2 + \|E_{k+1}\|_G^2) + C\tau^5.
\end{align*}

We summarize the above estimates:
$$
  \frac12\|E_{k+1}\|_G^2 - \frac12\|E_k\|_G^2
	\le  C_1\tau(\|E_{k}\|_G^2 + \|E_{k+1}\|_G^2) + C_2\tau^5.
$$
Since $(1+x)/(1-x)\le 1+4x$ for $0\le x\le 1/2$, we infer that 
for $\tau\le 1/(4C_1)$,
$$
  \|E_{k+1}\|_G^2 \le \frac{1+2C_1\tau}{1-2C_1\tau}\|E_k\|_G^2
	+ \frac{2C_2\tau^5}{1-2C_1\tau} 
	\le (1+8C_1\tau)\|E_k\|_G^2 + 4C_2\tau^5.
$$
Solving these recursive inequalities, it follows that
\begin{align*}
  \|E_{k+1}\|_G^2 &\le (1+8C_1\tau)^{k+1}\|E_0\|_G^2 
	+ 4C_2\tau^5\sum_{j=0}^k(1+8C_1\tau)^j \\
	&= (1+8C_1\tau)^{t_{k+1}/\tau}\|E_0\|_G^2 
	+ \frac{C_2\tau^4}{2C_1}((1+8C_1\tau)^{t_{k+1}/\tau}-1) \\
	&\le e^{8C_1 t_{k+1}}\|E_0\|_G^2 + \frac{C_2}{2C_1}e^{8C_1 t_{k+1}}\tau^4
	\le C\tau^4,
\end{align*}
where we have used that $\|E_0\|_G\le C(\|e_0\|+\|e_1\|)=C\|e_1\|\le C\tau^2$.
Therefore, $\|e_k\|^2+\|e_{k+1}\|^2 \le C\|E_k\|^2\le C\tau^4$.
Finally, taking into account \eqref{3.delta}, we find that
$$
  \|v_k-v(t_k)\| = \|e_k - \delta_I(t_k)\|
	\le \|e_k\| + \|\delta_I(t_k)\| \le C\tau^2, \quad k\ge 2.
$$
This finishes the proof.


\section{Numerical examples}\label{sec.numerics}

In this section, we present some numerical examples for the spatial one-dimensional
Shigesada-Kawasaki-Teramoto cross-diffusion system, which illustrate the 
time decay rate of the entropy functional. Numerical examples for the quantum
diffusion equation can be found in \cite{BEJ12}.
We choose the two-step BDF and $\gamma$-method in time, 
defined in Remark \ref{rem.ex}, and finite differences in space.

The grid is defined by $x_i=ih$, $i=0,\ldots,N$, and $t_k=k\tau$, $k\ge 0$, with
constant space step size $h=1/N>0$ and time step size $\tau>0$. 
In the numerical simulations, we have taken $h=0.005$ and $\tau=10^{-6}$.
We choose the initial datum
$u^{(1)}(x,0)=2e^{-x}\sin(2\pi x)+10$ and $u^{(2)}(x,0)=-4e^{-x}\sin(2\pi x)+10$
for $x\in(0,1)$. 

The operator $A=(A_1,A_2)$ is discretized in its formulation
$A_j(u)=-\diver(d_ju^{(j)}+\frac12 a_j\na(u^{(j)})^2+\na(u^{(1)}u^{(2)})$,
where $j=1,2$. Then the two-step BDF (or simpler BDF2) scheme for 
$v_{1,i}^k$, approximating $u^{(1)}(x_i,t_k)^{\alpha/2}$,
and $v_{2,i}^k$, approximating $u^{(2)}(x_i,t_k)^{\alpha/2}$, reads as follows:
\begin{align*}
  \frac{2}{\alpha}(v_{1,i}^k)^{\alpha/2-1}
	&\left(\frac32 v_{1,i}^k - 2v_{1,i}^{k-1} + \frac12v_{1,i}^{k-2}\right)
	- \frac{d_1\tau}{h^2}\left((v_{1,i+1}^k)^{\alpha/2} - 2(v_{1,i}^k)^{\alpha/2}
	+ (v_{1,i-1}^k)^{\alpha/2}\right) \\
	&{}- \frac{a_1\tau}{2h^2}\left((v_{1,i+1}^k)^{\alpha} - 2(v_{1,i}^k)^{\alpha}
	+ (v_{1,i-1}^k)^{\alpha}\right) \\
	&{}- \frac{\tau}{h^2}\left((v_{1,i+1}^k v_{2,i+1}^k)^{\alpha/2}
	- 2(v_{1,i}^k v_{2,i}^k)^{\alpha/2} + (v_{1,i-1}^k v_{2,i-1}^k)^{\alpha/2}\right)
	= 0, \\
  \frac{2}{\alpha}(v_{2,i}^k)^{\alpha/2-1}
	&\left(\frac32 v_{2,i}^k - 2v_{2,i}^{k-1} + \frac12v_{2,i}^{k-2}\right)
	- \frac{d_2\tau}{h^2}\left((v_{2,i+1}^k)^{\alpha/2} - 2(v_{2,i}^k)^{\alpha/2}
	+ (v_{1,i-1}^k)^{\alpha/2}\right) \\
	&{}- \frac{a_2\tau}{2h^2}\left((v_{2,i+1}^k)^{\alpha} - 2(v_{2,i}^k)^{\alpha}
	+ (v_{1,i-1}^k)^{\alpha}\right) \\
	&{}- \frac{\tau}{h^2}\left((v_{1,i+1}^k v_{2,i+1}^k)^{\alpha/2}
	- 2(v_{1,i}^k v_{2,i}^k)^{\alpha/2} + (v_{1,i-1}^k v_{2,i-1}^k)^{\alpha/2}\right)
	= 0,	
\end{align*}
where $i=1,\ldots,N-1$ and $k\ge 2$. To determine the discrete solution for $k=1$,
we employ the implicit Euler method. 
The periodic boundary conditions are $v_{j,0}=v_{j,N}$ and
$v_{j,1}=v_{j,N+1}$ for $j=1,2$. 
The above nonlinear system is solved by the Newton method.
The scheme using the $\gamma$-method (with $\gamma=1/5$) is defined in a similar way.
The parameters are chosen as follows:
$$
  \textrm{ Test A: }\  d_1 = d_2 = 1, \, a_1 = a_2 = 0.01, \quad
  \textrm{ Test B: }\  d_1 = d_2 = 1, \, a_1 = a_2 = 1. 
$$
In Test A, the self-diffusion parameters are small compared to the remaining
terms, whereas in Test B, all parameters, including the cross-diffusion terms,
are of the same order.

In Figure \ref{fig.0}, the time evolution of the population densities
$u^{(1)}$ and $u^{(2)}$ for the parameters according to Test B, $\alpha=3/2$,
computed from the BDF2 scheme, is illustrated. Because of the absence of source
terms and the periodic boundary conditions, the densities
converge to the constant steady state for large times.

\begin{figure}
\centering
\includegraphics[width=80mm]{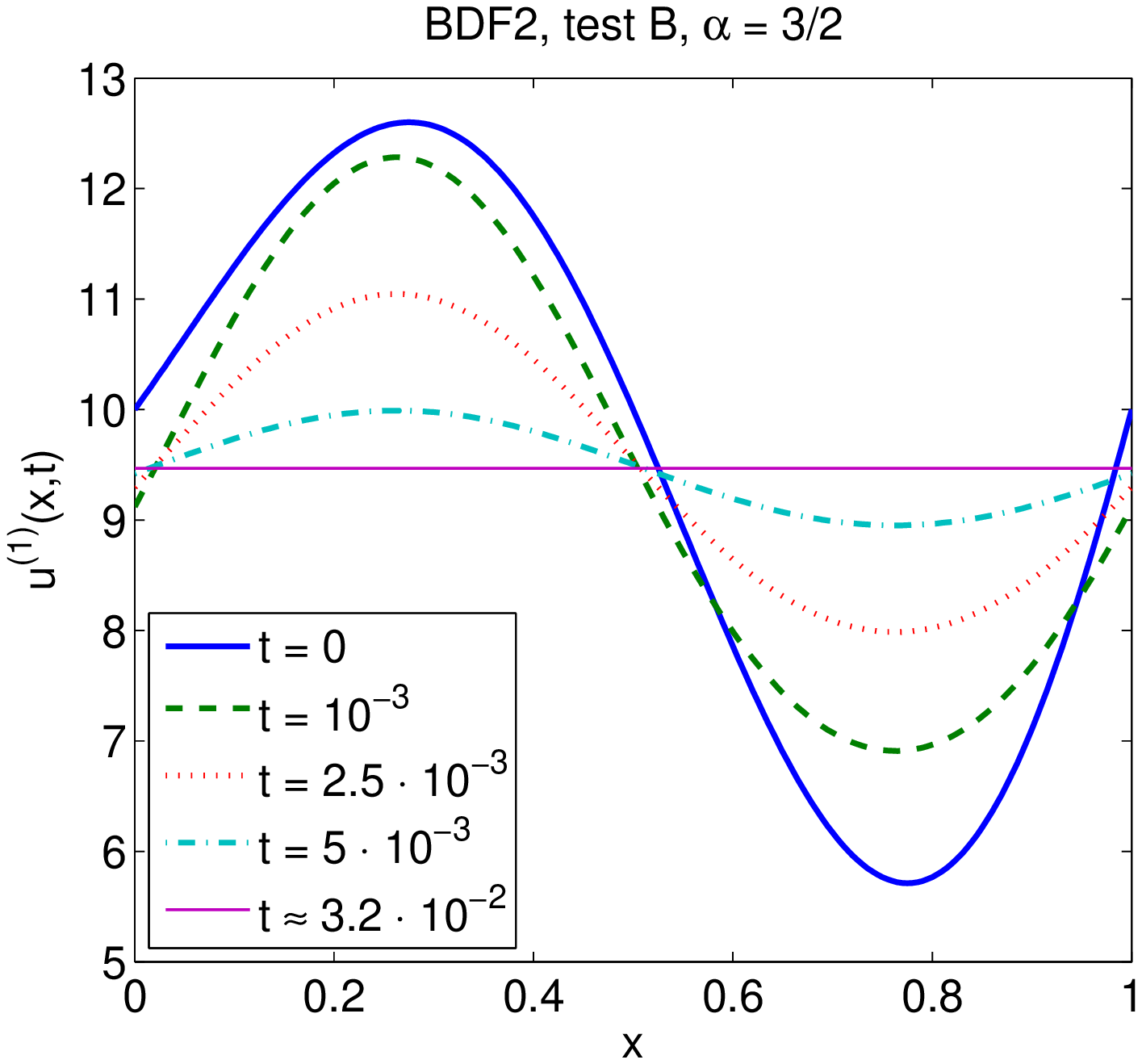}
\includegraphics[width=80mm]{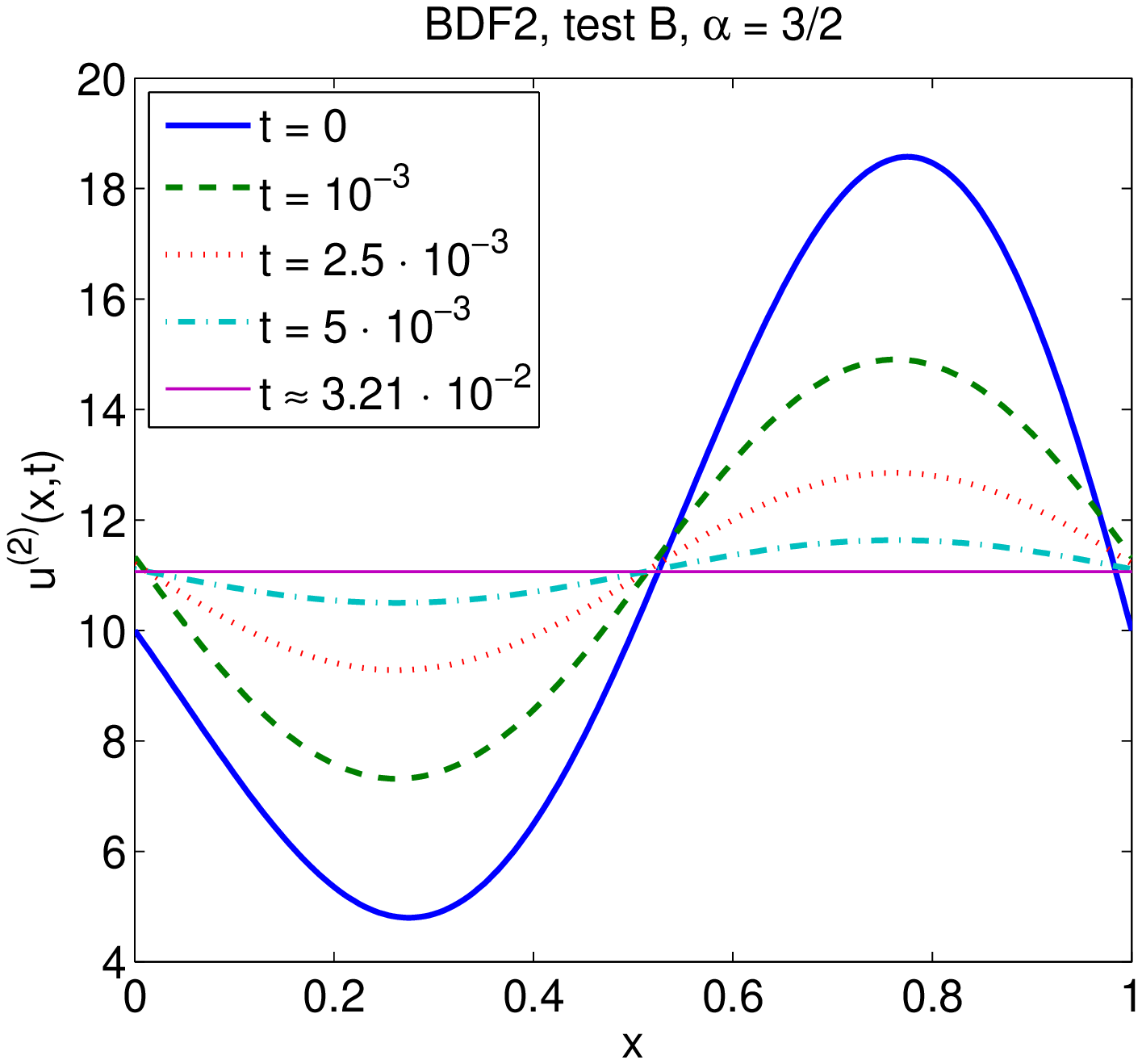}
\caption{Time evolution of the solution $u^{(1)}(x,t)$ and $u^{(2)}(x,t)$
to the population model computed from the BDF2 scheme (Test B, $\alpha=3/2$).}
\label{fig.0}
\end{figure}

The convergence of the scheme is shown in Figure \ref{fig.cr} 
at time $t_m=5\cdot 10^{-4}$. The error is measured in the $\ell^2$ norm
$$
  \|e_m\|_2 = \left(\sum_{i=0}^{N-1}\sum_{j=1}^2(v_{j,i}^m-V_{j,i}^m)^2 h\right)^{1/2},
$$
where $V_{j,i}^m$ is the reference solution computed by using 
the very small time step $\tau = 10^{-8}$.
The rates have been obtained by the linear regression method. 
As expected, the rate of convergence is (approximately) two, even for $\alpha=1$
which was excluded in our analysis.
The rate for $\alpha=2$ is the largest which comes from the fact that
in this case, we recover the usual BDF2 method without additional
nonlinearities (since $w_k=v_k$).

\begin{figure}
\centering
\includegraphics[width=80mm]{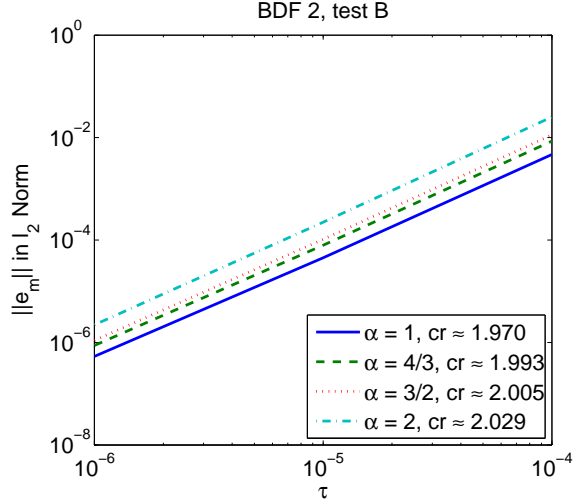}
\caption{Error $\|e_m\|_2$ versus the time step size for the BDF2 scheme 
at time $t_m=5\cdot 10^{-4}$.}
\label{fig.cr}
\end{figure}

Figures \ref{fig.bdf} and \ref{fig.gamma} illustrate the time decay of the discrete
relative entropy $E_{\alpha,d}^{\rm rel}=H_{\alpha,d}[V_k]-H_{\alpha,d}[V^*]$,
where
$$
  H_{\alpha,d}[V_k] = \frac12\sum_{i,j=0}^1 G_{ij}\sum_{\ell=1}^{N-1}
	(v_{1,\ell}^{k+i}v_{1,\ell}^{k+j}+v_{2,\ell}^{k+i}v_{2,\ell}^{k+j}),
$$
and $V^*$ represents the (constant) stationary solution. The coefficients of the 
matrix $(G_{ij})$ are given in Remark \ref{rem.ex}. We observe that in all
considered cases, the discrete entropy converges to the equilibrium 
with exponential rate.
	
\begin{figure}
\centering
\includegraphics[width=80mm]{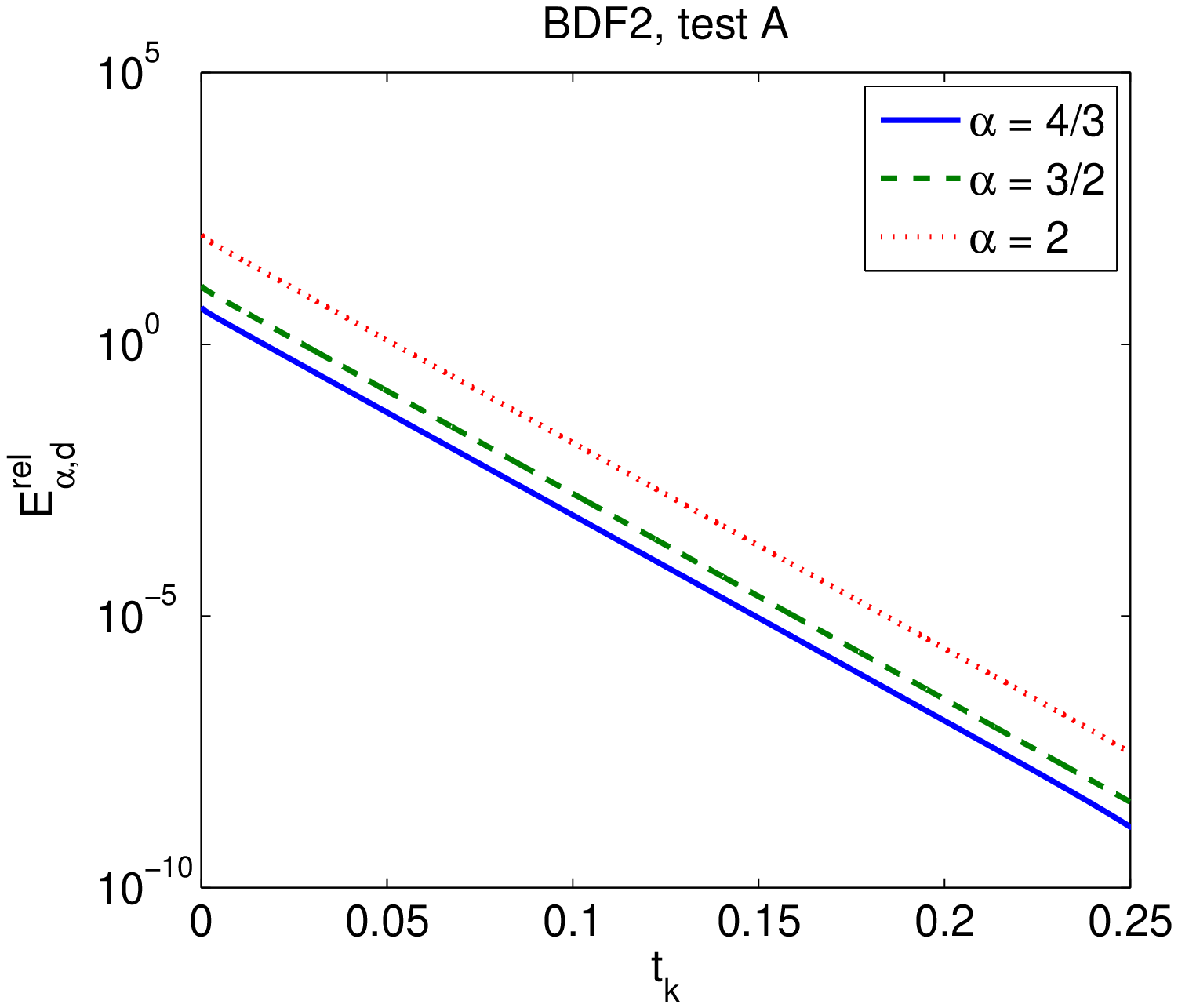}
\includegraphics[width=80mm]{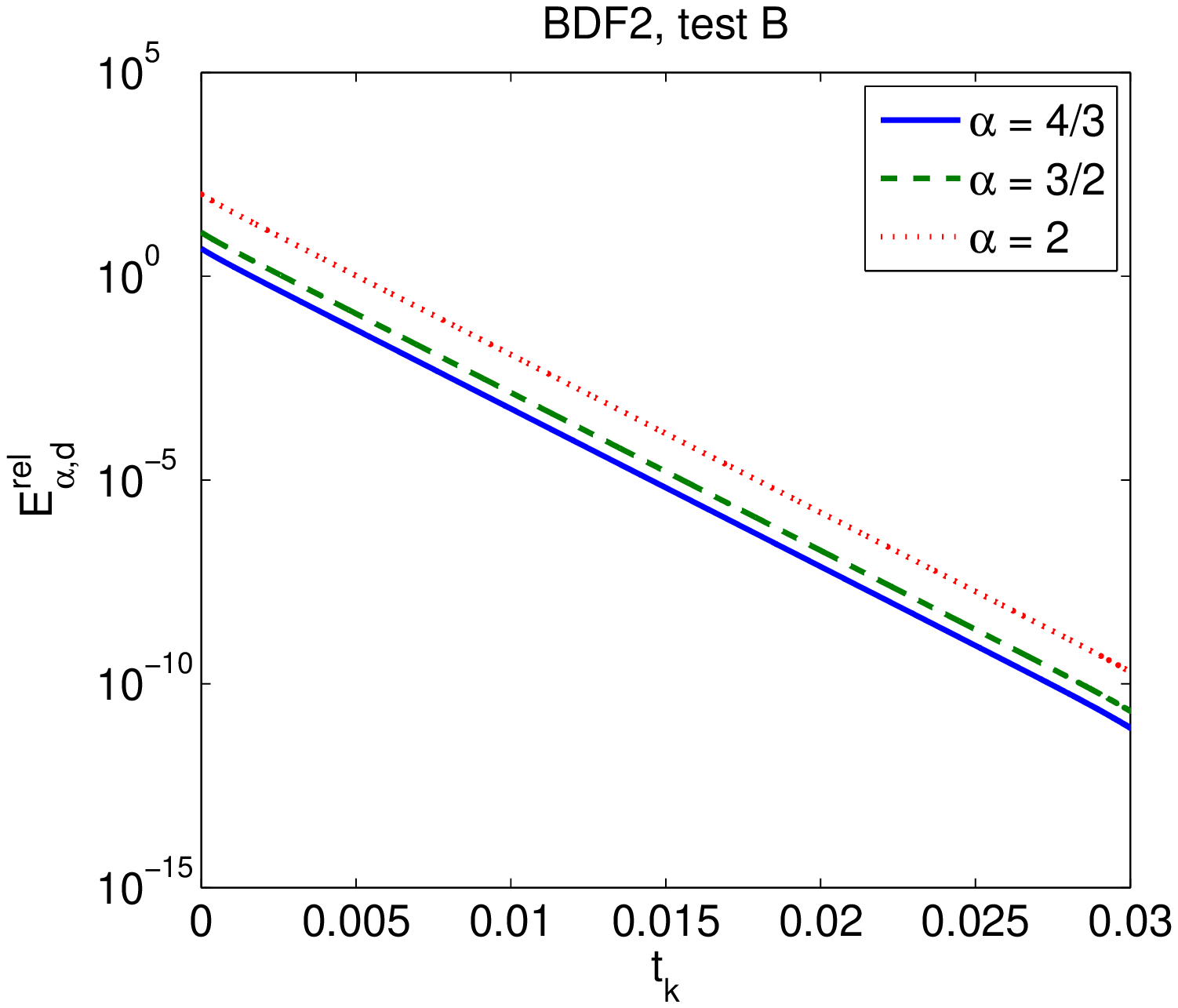}
\caption{The relative entropy for the BDF2 scheme versus time.}
\label{fig.bdf}
\end{figure} 

\begin{figure}
\centering
\includegraphics[width=80mm]{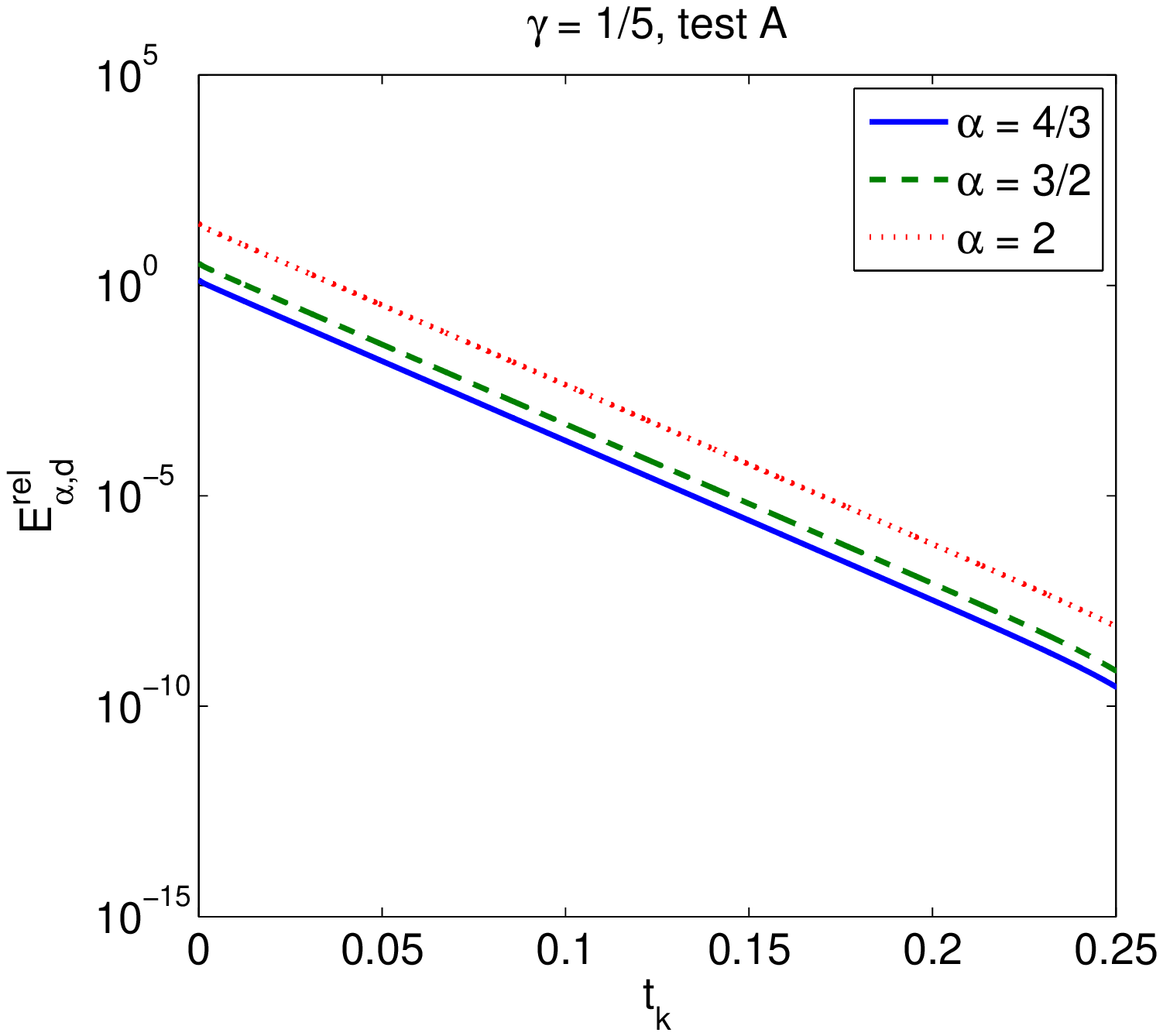}
\includegraphics[width=80mm]{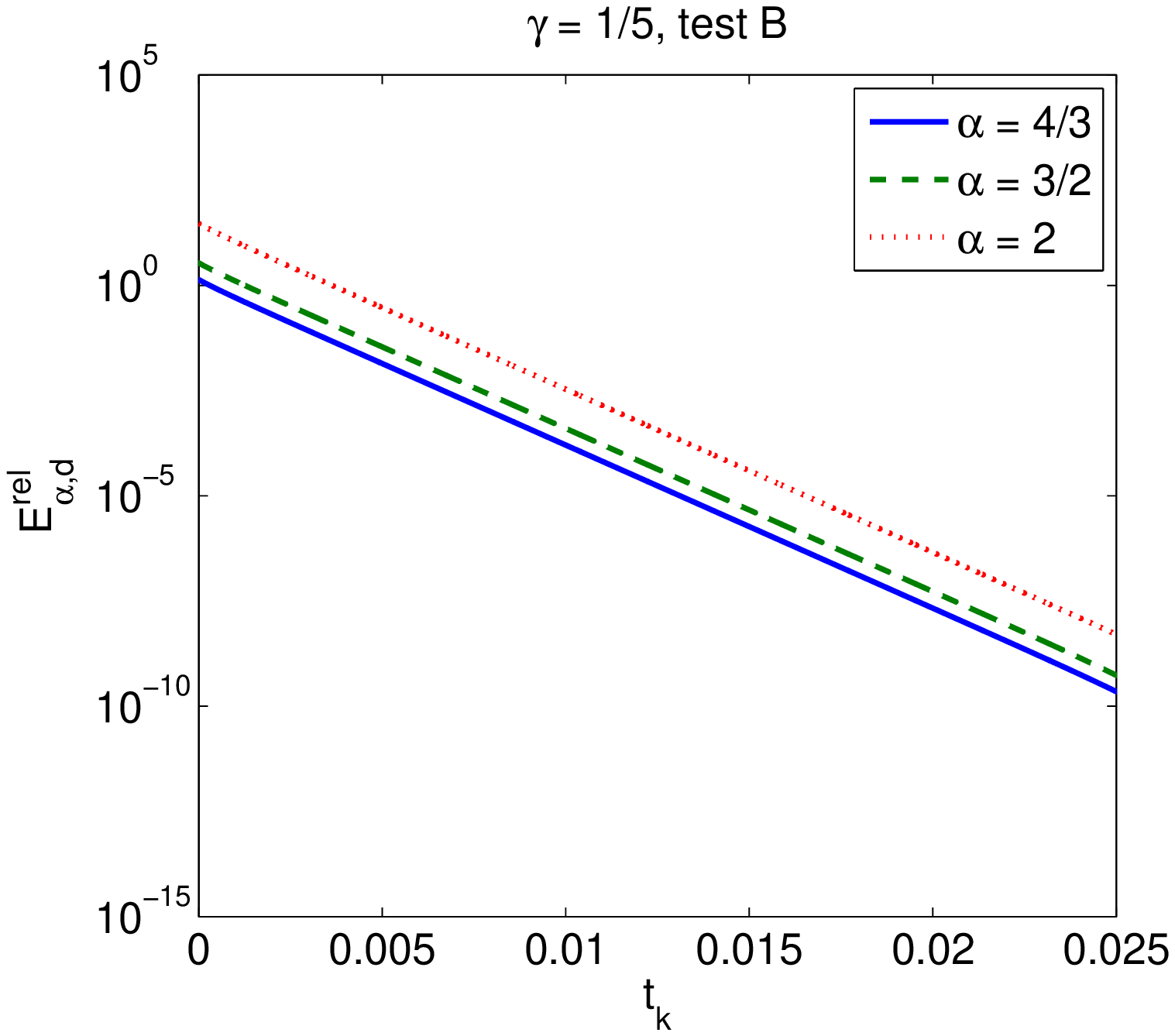}
\caption{The relative entropy for the $\gamma$-method versus time ($\gamma=1/5$).}
\label{fig.gamma}
\end{figure}


\begin{appendix}
\section{A family of second-order G-stable one-leg methods}

We derive all G-stable one-leg schemes which are of second order (in the
truncation error). Then $p=2$ and for $v=(v_0,v_1,v_2)^\top$,
$$
  \rho(E)v = \alpha_0v_0 + \alpha_1v_1 + \alpha_2v_2, \quad
  \sigma(E)v = \beta_0v_0 + \beta_1v_1 + \beta_2v_2.
$$
The normalization, consistency, and second-order accurate conditions 
(see Section \ref{sec.g.ene}) lead to the four equations
\begin{equation}\label{a.eqs1}
\begin{array}{ll}
  1 = \sigma(1) = \beta_0 + \beta_1 + \beta_2, &\quad
  0 = \rho(1) = \alpha_0 + \alpha_1 + \alpha_2, \\
  1 = \sigma(1) = \rho'(1) = \alpha_1 + 2\alpha_2, &\quad
  1+2\alpha_2 = \rho'(1)+\rho''(1) = 2\sigma'(1) = 2\beta_1+4\beta_2.
\end{array}
\end{equation}
The G-stability condition \eqref{g.stable} can be written as
\begin{equation}\label{a.stable}
  \rho(E)v\,\sigma(E)v - \frac12V_1^\top GV_1 + \frac12V_0^\top GV_0
  - (\gamma_0v_0+\gamma_1v_1+\gamma_2v_2)^2 = 0,
\end{equation}
where $V_1=(v_1,v_2)^\top$, $V_0=(v_0,v_1)^\top$, and $\gamma_j$ are some
real constants. This formulation
is possible since the occuring polynomials are at most quadratic and consequently,
the positive polynomial $p(v) = \rho(E)v\,\sigma(E)v - \frac12V_1^\top GV_1 
+ \frac12V_0^\top GV_0$ can be written as a single square. 
Condition \eqref{a.stable} has to hold for all $v\in\R^3$. 
Identifying the coefficients of $p(v)$ with those from 
$(\gamma_0 v_0+\gamma_1 v_1+\gamma_2 v_2)^2$, we find the following six equations:
\begin{equation}\label{a.eqs2}
\begin{array}{ll}
  0 = \alpha_0\beta_0 + G_{00} - \gamma_0^2, &\quad
  0 = 2G_{01} + \alpha_1\beta_0 + \alpha_0\beta_1 - 2\gamma_0\gamma_1, \\
  0 = \alpha_0\beta_2+\alpha_2\beta_0 - 2\gamma_0\gamma_2, &\quad
  0 = \gamma_{11}-\gamma_{11}+\alpha_1\beta_1-\gamma_1^2, \\
  0 = -2G_{01} + \alpha_2\beta_1+\alpha_1\beta_2-2\gamma_1\gamma_2, &\quad
  0 = -G_{11} + \alpha_2\beta_2-\gamma_2^2,
\end{array}
\end{equation}
where $G=(G_{ij})_{i,j=0,1}$. Observing that $G_{01}=G_{10}$, 
conditions \eqref{a.eqs1} and \eqref{a.eqs2} yield 10 equations for the 12 unknowns
$G_{00}$, $G_{01}$, $G_{11}$, $\alpha_j$, $\beta_j$, $\gamma_j$ ($i,j=0,1,2$). 
We also require the positive definiteness of the matrix $G$, i.e.
$$
  G_{00} > 0, \quad \det G = G_{00}G_{11}-G_{01}^2 > 0.
$$
Solving the nonlinear system \eqref{a.eqs1} and \eqref{a.eqs2} with the command
{\tt solve} in {\tt Maple} gives two sets of solutions. One solution set
yields a matrix $G$ with $\det G=0$ such that this solution can be excluded. The other
set is given by
\begin{equation}\label{a.scheme}
  (\alpha_0,\alpha_1,\alpha_2) = (\alpha_2-1,1-2\alpha_2,\alpha_2), \quad
  (\beta_0,\beta_1,\beta_2) = (\tfrac12-\alpha_2+\beta_2,\tfrac12+\alpha_2-2\beta_2,
  \beta_2),
\end{equation}
where $\alpha_2$ and $\beta_2$ are free parameters,
the matrix
$$
  G = \frac14\begin{pmatrix}
  (2\alpha_2-5)\alpha_2+2\beta_2+2 & (-2\alpha_2+3)\alpha_2-2\beta_2 \\
  (-2\alpha_2+3)\alpha_2-2\beta_2  & (2\alpha_2-1)\alpha_2+2\beta_2
  \end{pmatrix},
$$
and the constants $\gamma_j$ are solutions of certain quadratic equations
involving the coefficients $\alpha_j$ and $\beta_j$. 
The matrix $G$ is positive definite if and only if
$(2\alpha_2-5)\alpha_2+2\beta_2+2>0$ and $\det G=\beta_2-\alpha_2/2>0$.
The latter condition implies the former one since
$$
  (2\alpha_2-5)\alpha_2+2\beta_2+2 > (2\alpha_2-5)\alpha_2+\alpha_2+2 
  = 2(\alpha_2-1)^2 \ge 0.
$$
Consequently, all one-leg schemes \eqref{a.scheme} satisfying
$\beta_2> \alpha_2/2$ are of second order and G-stable.

The examples given in Remark \ref{rem.ex} are included in the above family of
schemes. Indeed, choosing $\alpha_2=\frac32$ and $\beta_2=1$, we find the
two-step BDF method, and setting for $\gamma\in(0,1]$ $\alpha_2=1/(\gamma+1)$
and $\beta_2=(3\gamma+1)/(2(\gamma+1)^2)$, we recover the $\gamma$-method
of \cite{DLN83,KuSh05}. 

Notice that we could repeat the same procedure to derive all first-order
G-stable schemes. The solution set will depend on three free parameters since
only 9 equations for 12 unknowns need to be solved. We leave the
details to the reader.
\end{appendix}


\end{document}